\documentclass[reqno,11pt,english]{amsart}
\usepackage[T1]{fontenc}
\usepackage[latin9]{inputenc}
\usepackage{amsthm}
\usepackage{amstext}
\usepackage{amssymb}
\usepackage{esint}

\makeatletter
\numberwithin{equation}{section}
\numberwithin{figure}{section}
\theoremstyle{plain}
\newtheorem{thm}{\protect\theoremname}[section]
  \theoremstyle{plain}
  \newtheorem{prop}[thm]{\protect\propositionname}
  \theoremstyle{plain}
  \newtheorem{lem}[thm]{\protect\lemmaname}
  \theoremstyle{plain}
  \newtheorem{cor}[thm]{\protect\corollaryname}

\usepackage{amsmath,amsfonts,amssymb,amsthm,epsfig}

\usepackage{color}

\usepackage[final,allcolors=blue,colorlinks=true]{hyperref}
\usepackage[leqno]{amsmath}

\def\R{\mathbb R}
\def\N{\mathbb N}

\def\wq{\infty}
\def\cal{\mathcal}
\def\al{\alpha}

\def\ga{\gamma}
\def\de{\delta}
\def\ep{\epsilon}
\def\la{\lambda}
\def\si{\sigma}

\def\var{\varphi}
\def\om{\omega}
\def\na{\nabla}

\def\Om{\Omega}
\def\De{\Delta}

\def\pa{\partial}

\newcommand{\dist}{{\rm dist}}
\newcommand{\lv}{\left\Vert}
\newcommand{\rv}{\right\Vert}
\newcommand\esssup{{\rm \,esssup\,}}

\numberwithin{equation}{section}
\theoremstyle{definition}
\begingroup

\endgroup

\makeatother

\usepackage{babel}
  \providecommand{\corollaryname}{Corollary}
  \providecommand{\lemmaname}{Lemma}
  \providecommand{\propositionname}{Proposition}
\providecommand{\theoremname}{Theorem}

\begin{document}

\title{\title[Infinitely Many Solutions to  Quasilinear Equations] IInfinitely
many solutions for quasilinear elliptic equations involving double
critical terms and boundary geometry}

\author{Chun-Hua Wang, Chang-Lin Xiang}

\date{\today}

\address{[Chun-Hua Wang] School of Mathematics and Statistics, Central China
Normal University, Wuhan 430079, P. R. China. }

\email{[Chun-Hua Wang] chunhuawang@mail.ccnu.edu.cn}

\address{[Chang-Lin Xiang] Department of Mathematics and Statistics, P.O.
Box 35 (MaD) FI-40014 University of Jyv\"askyl\"a, Finland}

\email{[Chang-Lin Xiang] changlin.c.xiang@jyu.fi}
\begin{abstract}
Let $1<p<N$, $p^{*}=Np/(N-p)$, $0<s<p$, $p^{*}(s)=(N-s)p/(N-p)$,
and $\Om\in C^{1}$ be a bounded domain in $\R^{N}$ with $0\in\bar{\Om}.$
In this paper, we study the following problem
\[
\begin{cases}
-\Delta_{p}u=\mu|u|^{p^{*}-2}u+\frac{|u|^{p^{*}(s)-2}u}{|x|^{s}}+a(x)|u|^{p-2}u, & \text{in }\Om,\\
u=0, & \text{on }\pa\Om,
\end{cases}
\]
where $\mu\ge0$ is a constant, $\De_{p}$ is the $p$-Laplacian operator
and $a\in C^{1}(\bar{\Om})$. By an approximation argument, we prove
that if $N>p^{2}+p,a(0)>0$ and $\Omega$ satisfies some geometry
conditions if $0\in\partial\Omega$, say, all the principle curvatures
of $\partial\Omega$ at $0$ are negative, then the above problem
has infinitely many solutions.
\end{abstract}

\maketitle
{\small

\tableofcontents{}

\section{Introduction and main results}

Let $1<p<N$, $p^{*}=Np/(N-p)$, $0<s<p$, $p^{*}(s)=(N-s)p/(N-p)$,
and $\Om\in C^{1}$ be an open bounded domain in $\R^{N}$ with $0\in\bar{\Om}.$
In this paper, we study the following quasilinear elliptic equations
\begin{equation}
\begin{cases}
-\Delta_{p}u=\mu|u|^{p^{*}-2}u+\frac{|u|^{p^{*}(s)-2}u}{|x|^{s}}+a(x)|u|^{p-2}u, & \text{in }\Om,\\
u=0, & \text{on }\pa\Om,
\end{cases}\label{1.1}
\end{equation}
where $\mu\ge0$ is a constant,
\[
\Delta_{p}u=\sum_{i=1}^{N}\partial_{x_{i}}(|\nabla u|^{p-2}\partial_{x_{i}}u),\quad\nabla u=(\partial_{x_{1}}u,\cdots,\partial_{x_{N}}u)
\]
is the $p$-Laplacian operator and $a\in C^{1}(\bar{\Om})$.

The functional corresponding to equation (\ref{1.1}) is
\begin{equation}
I(u)=\frac{1}{p}\int_{\Omega}\big(|\nabla u|^{p}-a(x)|u|^{p}\big)dx-\frac{\mu}{p^{*}}\int_{\Omega}|u|^{p^{*}}dx-\frac{1}{p^{*}(s)}\int_{\Omega}\frac{|u|^{p^{*}(s)}}{|x|^{s}}dx,\label{1.2}
\end{equation}
for $u\in W_{0}^{1,p}(\Om)$. All of the integrals in energy functional
$I$ are well defined, due to the Sobolev inequality
\begin{eqnarray*}
C\left(\int_{\mathbb{R}^{N}}|\varphi|^{p^{*}}dx\right)^{\frac{p}{p^{*}}}\leq\int_{\mathbb{R}^{N}}|\nabla\varphi|^{p}dx, &  & \forall\:\varphi\in W_{0}^{1,p}(\Om),
\end{eqnarray*}
for $C=C(N,p)>0$, and due to the Caffarelli-Kohn-Nirenberg inequality
(see \cite{Caffarelli-Kohn-Nirenberg1984})
\begin{eqnarray*}
C\left(\int_{\Omega}\frac{|\var|^{p^{*}(s)}}{|x|^{s}}dx\right)^{\frac{p}{p^{*}(s)}}\leq\int_{\mathbb{R}^{N}}|\nabla\varphi|^{p}dx, &  & \forall\:\varphi\in W_{0}^{1,p}(\Om),
\end{eqnarray*}
for $C=C(N,p,s)>0$.

Since the pioneer work of Br\'ezis and Nirenberg \cite{BrezisNirenberg1983},
there are enormous results on semilinear problems e.g. \cite{BahriCoron1998,Brezis1985,Cappozzi1985,Cerami1986,Coron1984,DevillanovaSolomini2002,Struwe1984}
and on quasilinear problems e.g. \cite{CaoPengYan2012,CingolaniVannella2009,CingolaniVannella2009-2,DegiovanniLancelotti2009,Egnell1988,GhoussoubYuan2000,GueddaVeron1989,Papageorgiou2007,ZhuXP1988}
with Sobolev exponents.

Without the presence of the Hardy term $|x|^{-s}|u|^{p^{*}(s)-2}u$
in equation (\ref{1.1}), Devillanova and Solimini \cite{DevillanovaSolomini2002}
considered equation (\ref{1.1}) in the semilinear case ($p=2$).
With the assumptions that $\mu>0$ and $a\equiv\la$ in $\Om$ for
some constant $\la>0$, they proved the existence of infinitely many
solutions to equation (\ref{1.1}) if $N>6$. Then Cao, Peng and Yan
\cite{CaoPengYan2012} generalized their result to the quasilinear
case, that is, $1<p<N$. Under the same assumptions on $\mu$ and
$a$ as that of Devillanova and Solimini \cite{DevillanovaSolomini2002},
they proved the existence of infinitely many solutions to equation
(\ref{1.1}) if $N>p^{2}+p$.

In the presence of the Hardy term $|x|^{-s}|u|^{p^{*}(s)-2}u$ in
equation (\ref{1.1}), Yan and Yang \cite{YanYang2013} considered
equation (\ref{1.1}) in the semilinear case. Under the assumption
that $a(0)>0$ and the following geometry assumption imposed on $\Om$:
$\Om\in C^{3}$ and
\begin{equation}
\text{ all the principle curvatures of }\ensuremath{\pa\Om}\text{ at }0\text{ are negative }\text{if }0\in\pa\Om,\label{condition of YY}
\end{equation}
they proved the existence of infinitely many solutions for equation
(\ref{1.1}) if $N>6$.

So a natural problem is whether in the quasilinear case equation (\ref{1.1})
has infinitely many solutions. The functional $I$ defined by (\ref{1.2})
does not satisfy the Palais-Smale condition at large energy level.
So it is impossible to apply the mountain pass lemma \cite{AmbrosettiRabinowitz1973}
directly to obtain the existence of infinitely many solutions for
equation (\ref{1.1}). In this paper, we follow the idea of Devillanova
and Solimini \cite{DevillanovaSolomini2002} to study the following
perturbed problem:
\begin{equation}
\begin{cases}
-\Delta_{p}u=\mu|u|^{p^{*}-2-\ep}u+\frac{|u|^{p^{*}(s)-2-\ep}u}{|x|^{s}}+a(x)|u|^{p-2}u, & \text{in }\Om,\\
u=0, & \text{on }\pa\Om,
\end{cases}\label{1.3}
\end{equation}
where $\ep>0$ is a small constant. See also \cite{CaoPengYan2012,CaoYan2010,YanYang2013}
for applications of the same idea. The functional corresponding to
equation (\ref{1.3}) is
\begin{equation}
I_{\ep}(u)=\frac{1}{p}\int_{\Omega}\big(|\nabla u|^{p}-a(x)|u|^{p}\big)dx-\frac{\mu}{p^{*}-\ep}\int_{\Omega}|u|^{p^{*}-\ep}dx-\frac{1}{p^{*}(s)-\ep}\int_{\Omega}\frac{|u|^{p^{*}(s)-\ep}}{|x|^{s}}dx,\label{1.4}
\end{equation}
for $u\in W_{0}^{1,p}(\Om)$.

Now $I_{\ep}$ is an even functional and satisfies the Palais-Smale
condition in all energy levels. It follows from the symmetric mountain
pass lemma \cite{AmbrosettiRabinowitz1973,Rabinowitz1986} that equation
(\ref{1.3}) has infinitely many solutions. See also \cite{GhoussoubYuan2000,ShenGuo1987}.
Precisely, for $\ep>0$ fixed, there are positive numbers $c_{\ep,l}$
and critical points $u_{\ep,l}$, $l=1,2,\cdots,$ such that
\begin{eqnarray*}
I(u_{\ep,l})=c_{\ep,l}\to\wq, &  & \text{as }l\to\wq.
\end{eqnarray*}
Moreover, for each $l\ge1$ fixed, the sequence $\{c_{\ep,l}\}_{\ep>0}$
is bounded with respect to $\ep$ and thus can be assumed to converge
to a limit $c_{l}$ as $\ep\to0$.

To obtain the existence of infinitely many solutions for equation
(\ref{1.1}), the first step is to investigate whether $u_{\ep,l}$
converges strongly in $W_{0}^{1,p}(\Om)$ as $\ep\to0$. That is,
we need to study the compactness of the set of solutions for equation
(\ref{1.3}) for all $\ep>0$ small. If $u_{\ep,l}$ is proved to
converge to some $u_{l}\in W_{0}^{1,p}(\Om)$ strongly in $W_{0}^{1,p}(\Om)$,
then the next step is to investigate whether $c_{l}\to\wq$ as $l\to\wq$.
If so, then we obtain infinitely many solutions for equation (\ref{1.1})
with arbitrarily large energy level.

Throughout the paper, we use $\|\cdot\|$ to denote the norm of $W_{0}^{1,p}(\Om)$.
We assume that $\Om\in C^{1}$ satisfies the following condition:
\begin{equation}
x\cdot\nu\le0\text{ in a neighborhood of \ensuremath{0}\,\ in }\pa\Om\text{ if }0\in\pa\Om,\label{condition of WX}
\end{equation}
where $\nu$ is the outward unit normal of $\pa\Om$. Our main result
in this paper is the following theorem.
\begin{thm}
\label{thm: main result} Suppose that $a(0)>0$ and $\Om\in C^{1}$
satisfies the condition (\ref{condition of WX}). If $N>p^{2}+p$,
then for any $u_{n}$ $(n=1,2,\cdots),$ which is a solution to equation
(\ref{1.3}) with $\ep=\ep_{n}\to0$, satisfying $||u_{n}||\le C$
for some constant $C$ independent of $n$, $u_{n}$ converges strongly
in \textup{$W_{0}^{1,p}(\Om)$ up to a subsequence as $n\to\wq$.}
\end{thm}
As an application of Theorem \ref{thm: main result}, we have the
following existence result for equation (\ref{1.1}).
\begin{thm}
\label{thm: infinitely many solutions} Suppose that $a(0)>0$ and
$\Om\in C^{1}$ satisfies the condition (\ref{condition of WX}).
If $N>p^{2}+p$, then equation (\ref{1.1}) has infinitely many solutions.
\end{thm}
We remark that Theorem \ref{thm: infinitely many solutions} generalizes
the result of Yan and Yang \cite{YanYang2013} from the semilinear
case of equation (\ref{1.1}) to the quasilinear case. To see this,
one only needs to observe that condition (\ref{condition of YY})
is just a special case of condition (\ref{condition of WX}). Indeed,
suppose that $\Om\in C^{3}$ and $0\in\pa\Om$ such that the condition
(\ref{condition of YY}) is satisfied. Then up to a rotation, we can
find a small constant $\de>0$ and a function $\var\in C^{3}$ such
that
\begin{eqnarray}
\Om\cap B_{\de}(0)=\{x\in\R^{N};x_{N}>\var(x^{\prime})\} & \text{ and } & \pa\Om\cap B_{\de}(0)=\{x\in\R^{N};x_{N}=\var(x^{\prime})\}.\label{eq: boundary representation}
\end{eqnarray}
Here we write $x=(x^{\prime},x_{N})\in\R^{N-1}\times\R$. By (\ref{condition of YY})
and (\ref{eq: boundary representation}), $\var$ satisfies the following
Taylor expansion at $x^{\prime}=0$:
\begin{eqnarray*}
\var(x^{\prime})=-\sum_{j=1}^{N-1}\al_{j}x_{j}^{2}+O(|x^{\prime}|^{3}), &  & \text{for }|x^{\prime}|\text{ small enough},
\end{eqnarray*}
with constants $\al_{j}>0$, $j=1,\cdots,N-1$. Then the outward unit
normal $\nu$ of $\pa\Om$ is given by
\begin{eqnarray*}
\nu(x)=\frac{(\pa_{x_{1}}\var(x^{\prime}),\cdots,\pa_{x_{N-1}}\var(x^{\prime}),-1)}{\sqrt{1+\sum_{j=1}^{N-1}|\pa_{x_{j}}\var(x^{\prime})|^{2}}}, &  & \text{for }x\in\pa\Om\cap B_{\de}(0).
\end{eqnarray*}
Thus
\begin{eqnarray*}
x\cdot\nu(x)=-\sum_{j=1}^{N-1}\al_{j}x_{j}^{2}+O(|x^{\prime}|^{3}) &  & \text{for }x\in\pa\Om\cap B_{\de}(0),
\end{eqnarray*}
which implies that
\begin{eqnarray*}
x\cdot\nu(x)\le0 &  & \text{for }x\in\pa\Om\cap B_{\de^{\prime}}(0),
\end{eqnarray*}
for some $0<\de^{\prime}<\de$. That is, (\ref{condition of WX})
is satisfied. So we find that condition (\ref{condition of YY}) is
a special case of condition (\ref{condition of WX}).

On the other hand, condition (\ref{condition of WX}) does allow more
possibilities than that of condition (\ref{condition of YY}). As
an example, suppose that $\pa\Om$ has a piece of concave boundary
close to $0$ if $0\in\pa\Om$. Precisely, let $\var\in C^{1}$ be
given such that (\ref{eq: boundary representation}) holds, and
\[
0=\var(0)\le\varphi(x^{\prime})+\sum_{j=1}^{N-1}\pa_{x_{j}}\var(x^{\prime})(0-x_{j}^{\prime})
\]
 for $x^{\prime}$ close to $0$. Then we have
\[
x\cdot\nu(x)=-\frac{\varphi(x^{\prime})+\sum_{j=1}^{N-1}\pa_{x_{j}}\var(x^{\prime})(0-x_{j}^{\prime})}{\sqrt{1+\sum_{j=1}^{N-1}|\pa_{x_{j}}\var(x^{\prime})|^{2}}}\le0
\]
for $x^{\prime}$ close to $0$. That is, (\ref{condition of WX})
is satisfied. In particular, if $\Om$ has a piece of flat boundary
in a neighborhood of $0$ when $0\in\pa\Om$, then all the principle
curvatures of $\pa\Om$ vanish at the point $0$. So in this case
(\ref{condition of WX}) is satisfied while (\ref{condition of YY})
is not satisfied.

Finally, we point out that in the case when $0\in\pa\Om$, the mean
curvature of $\pa\Om$ at $0$ plays an important role in the existence
of the mountain pass solutions to equation (\ref{1.1}). See for example
\cite{LinCS2010-1,GhoussoubKang2004,GhoussoubRobert2006,LinCS2010,LiLin2012}.

Our paper is organized as follows. In Section 2 we obtain some integral
estimates. In Section 3 we obtain estimates for solutions of equation
(\ref{1.3}) in the region which is close to but is suitably away
from the blow up point. We prove Theorems \ref{thm: main result}
and \ref{thm: infinitely many solutions} in Section 4. In order to
give a clear line of our framework, we will list some necessary estimates
on solutions of quasilinear equation with Hardy potential in Appendix
A, a decay estimate for critical Sobolev growth equation in Appendix
B, some estimates on solutions of $p$-Laplacian equation by Wolff
potential in Appendix C, and a global compactness result for the solution
$u_{n}$ of equation (\ref{1.3}) in Appendix D, respectively.

Our notations are standard. $B_{R}(x)$ is the open ball in $\R^{N}$
centered at $x$ with radius $R$. We write
\[
\fint_{E}udx=\frac{1}{|E|}\int_{E}udx,
\]
whenever $E$ is a measurable set with $0<|E|<\infty$, the $n$-dimensional
Lebesgue measure of $E$. Let $D$ be an arbitrary domain in $\R^{N}$.
We denote by $C_{0}^{\infty}(D)$ the space of smooth functions with
compact support in $D$. For any $1\le r\le\infty$, $L^{r}(D)$ is
the Banach space of Lebesgue measurable functions $u$ such that the
norm
\[
||u||_{r,D}=\begin{cases}
\left(\int_{D}|u|^{r}\right)^{\frac{1}{r}}, & \text{if }1\le r<\infty\\
\esssup_{D}|u|, & \text{if }r=\infty
\end{cases}
\]
is finite. The local space $L_{\text{loc}}^{r}(D)$ consists of functions
belonging to $L^{r}(D^{\prime})$ for all $D^{\prime}\subset\subset D$.
We also denote $d\mu_{s}=|x|^{-s}dx$ and $\|v\|_{q,\mu_{s}}=\left(\int|v|^{q}d\mu_{s}\right)^{1/q}$
when there is no confusion on the domain of the integral. A function
$u$ belongs to the Sobolev space $W^{1,r}(D)$ if $u\in L^{r}(D)$
and its first order weak partial derivatives also belong to $L^{r}(D)$.
We endow $W^{1,r}(D)$ with the norm
\[
||u||_{1,r,D}=||u||_{r,D}+||\na u||_{r,D}.
\]
The local space $W_{\text{loc}}^{1,r}(D)$ consists of functions belonging
to $W^{1,r}(D^{\prime})$ for all open $D^{\prime}\subset\subset D$.
We recall that $W_{0}^{1,r}(D)$ is the completion of $C_{0}^{\infty}(D)$
in the norm $||\cdot||_{1,r,D}$. For the properties of the Sobolev
functions, we refer to the monograph \cite{Ziemer}.

\section{Integral estimates }

Let $u_{n}$, $n=1,2,\ldots,$ be a solution of equation (\ref{1.3})
with $\ep=\ep_{n}\to0$, satisfying $\|u_{n}\|\le C$ for some constant
$C$ independent of $n$. In this section we deduce some integral
estimates for $u_{n}$. For any function $u,$ we define
\begin{equation}
\rho_{x,\la}(u)=\la^{\frac{N-p}{p}}u(\la(\cdot-x))\label{def: scaling transform}
\end{equation}
for any $\la>0$ and $x\in\R^{N}$. By Proposition \ref{prop:D.1},
$u_{n}$ can be decomposed as
\[
u_{n}=u_{0}+\sum_{j=1}^{m}\rho_{x_{n,j},\la_{n,j}}(U_{j})+\omega_{n}.
\]
Here $x_{n,j}=0$ for $j=k+1,\ldots,m$.

To prove that $u_{n}$ strongly converges in $W_{0}^{1,p}(\Om)$,
we only need to show that the bubbles $\rho_{x_{n,j},\la_{n,j}}(U_{j})$
will not appear in the decomposition of $u_{n}$. Among all the bubbles,
we can choose one bubble such that this bubble has the slowest concentration
rate. That is, the corresponding $\la$ is the lowest order infinity
among all the $\la$ appearing in the bubbles. For simplicity, we
denote by $\la_{n}$ the slowest concentration rate and by $x_{n}$
the corresponding concentration point.

For any $q>1$, denote
\[
\|u\|_{*,q}=\Big(\int_{\Omega}|u|^{q}dx\Big)^{\frac{1}{q}}+\Big(\int_{\Omega}|u|^{\frac{(N-s)q}{N}}d\mu_{s}\Big)^{\frac{N}{(N-s)q}}
\]
and $q'=\frac{q}{q-1}.$ Recall that $d\mu_{s}=|x|^{-s}dx$.

For any $p^{*}/p'<p_{2}<p^{*}<p_{1}$, $\al>0$ and $\la\ge1$, we
consider the following relation:
\begin{equation}
\begin{cases}
\|u_{1}\|_{*,p_{1}}\leq\alpha,\\
\|u_{2}\|_{*,p_{2}}\leq\alpha\la^{\frac{N}{p^{*}}-\frac{N}{p_{2}}},
\end{cases}\label{eq: 2.1}
\end{equation}
and define
\begin{equation}
\|u\|_{\ast,p_{1},p_{2},\la}=\inf\al,\label{eq: 2.2}
\end{equation}
where the infimum is taken over all $\al>0$ for which there exist
$u_{1},u_{2}$ such that $|u|\le u_{1}+u_{2}$ and (\ref{eq: 2.1})
holds. Our main result in this section is the following estimate.
\begin{prop}
\label{prop: main result in section 2} Let $u_{n}$, $n=1,2,\ldots,$
be a solution of equation (\ref{1.3}) with $\ep=\ep_{n}\to0$, satisfying
$||u_{n}||\le C$ for some positive constant $C$ independent of $n$.
Then for any $p_{1},p_{2}\in(p^{*}/p^{\prime},\wq)$, $p_{2}<p^{*}<p_{1}$,
there exists a constant $C=C(p_{1},p_{2})>0$, independent of $n$,
such that
\[
\|u_{n}\|_{\ast,p_{1},p_{2},\la_{n}}\le C
\]
 for all $n$. Here $\la_{n}$ is the slowest concentration rate of
$u_{n}$.
\end{prop}
Several lemmas are needed to prove Proposition \ref{prop: main result in section 2}.
In the rest of this section, let us fix a bounded domain $D$ with
$\Om\subset\subset D$ and define $r=\frac{1}{3}\dist(\Om,\pa D)$.
\begin{lem}
\label{lem: 2.2} Let $w\in W_{0}^{1,p}(D)$, $w\ge0$, be the solution
of
\begin{equation}
\left\{ \begin{array}{ll}
-\Delta_{p}w=\big(a_{1}(x)+\frac{a_{2}(x)}{|x|^{s}}\big)v^{p-1}, & \text{in }D,\\
w=0, & \text{on }\pa D,
\end{array}\right.\label{eq: 2.2.1}
\end{equation}
where $a_{1},a_{2},v\geq0$ are bounded functions in $D$. Then for
any $\frac{p^{*}}{p'}<p_{2}<p^{*}<p_{1},$ there is a constant $C=C(p_{1},p_{2})>0$,
such that for any $\la\geq1$,
\begin{equation}
\|w\|_{*,p_{1},p_{2},\la}\leq C\left(||a_{1}||_{\frac{N}{p}}+||a_{2}||_{\frac{N-s}{p-s},\mu_{s}}\right)^{\frac{1}{p-1}}||v||_{*,p_{1},p_{2},\la}.\label{eq: 2.2.1-1}
\end{equation}
\end{lem}
\begin{proof}
Let $\al>||v||_{*,p_{1},p_{2},\la}$ be an arbitrary constant. Then
by the definition of $||v||_{*,p_{1},p_{2},\la}$, there exist $v_{1},v_{2}$
such that $|v|\le v_{1}+v_{2}$ and (\ref{eq: 2.1}) holds with $u_{i}=v_{i}$,
$i=1,2$.

Let $w_{i}\in W_{0}^{1,p}(D)$, $w_{i}\ge0$, $i=1,2$, be the solution
of equation (\ref{eq: 2.2.1}) with $v=2v_{i}$. Then Corollary \ref{cor: A.2}
implies that
\begin{equation}
\|w_{i}\|_{\ast,p_{i}}\le C\left(||a_{1}||_{\frac{N}{p}}+||a_{2}||_{\frac{N-s}{p-s},\mu_{s}}\right)^{\frac{1}{p-1}}||v_{i}||_{*,p_{i}}.\label{eq:2.2.2}
\end{equation}

Let $\tilde{w}\in W_{0}^{1,p}(D)$, $\tilde{w}\ge0$, be the solution
of equation
\[
\left\{ \begin{array}{ll}
-\Delta_{p}w=\big(a_{1}(x)+\frac{a_{2}(x)}{|x|^{s}}\big)\left((2v_{1})^{p-1}+(2v_{2})^{p-1}\right), & \text{in }D,\\
w=0, & \text{on }\pa D.
\end{array}\right.
\]
Applying Corollary \ref{cor: A.2} gives us
\[
\begin{aligned}\|\tilde{w}\|_{\ast,p_{2}} & \le C\left(||a_{1}||_{\frac{N}{p}}+||a_{2}||_{\frac{N-s}{p-s},\mu_{s}}\right)^{\frac{1}{p-1}}\lv\left((2v_{1})^{p-1}+(2v_{2})^{p-1}\right)^{\frac{1}{p-1}}\rv_{\ast,p_{2}}\\
 & \le C\left(||a_{1}||_{\frac{N}{p}}+||a_{2}||_{\frac{N-s}{p-s},\mu_{s}}\right)^{\frac{1}{p-1}}\left(||v_{1}||_{*,p_{2}}+||v_{2}||_{*,p_{2}}\right)\\
 & \le C\left(||a_{1}||_{\frac{N}{p}}+||a_{2}||_{\frac{N-s}{p-s},\mu_{s}}\right)^{\frac{1}{p-1}}\al.
\end{aligned}
\]
Thus for any $x\in\Om,$ we have
\begin{equation}
\inf_{B_{r}(x)}\tilde{w}\le\left(\fint_{B_{r}(x)}\tilde{w}^{p_{2}}dy\right)^{\frac{1}{p_{2}}}\le C\left(||a_{1}||_{\frac{N}{p}}+||a_{2}||_{\frac{N-s}{p-s},\mu_{s}}\right)^{\frac{1}{p-1}}\al.\label{eq: 2.2.3}
\end{equation}

Note that $v^{p-1}\le((2v_{1})^{p-1}+(2v_{2})^{p-1})$. Thus $w\le\tilde{w}$
by comparison principle. Applying Proposition \ref{prop: C.1} gives
us
\begin{eqnarray*}
w(x)\le\tilde{w}(x)\le C\inf_{B_{r}(x)}\tilde{w}+Cw_{1}(x)+Cw_{2}(x), &  & \forall\, x\in\Om.
\end{eqnarray*}
Let $\tilde{w}_{1}(x)=C\inf_{B_{r}(x)}\tilde{w}+Cw_{1}(x)$ and $\tilde{w}_{2}(x)=Cw_{2}(x)$
for $x\in\Om$. Then $w\le\tilde{w}_{1}+\tilde{w}_{2}$ in $\Om$.
By (\ref{eq:2.2.2}) and (\ref{eq: 2.2.3}), we have that
\[
\|\tilde{w}_{1}\|_{\ast,p_{1}}\le C\left(||a_{1}||_{\frac{N}{p}}+||a_{2}||_{\frac{N-s}{p-s},\mu_{s}}\right)^{\frac{1}{p-1}}\al,
\]
and that
\[
\|\tilde{w}_{2}\|_{\ast,p_{2}}\le C\left(||a_{1}||_{\frac{N}{p}}+||a_{2}||_{\frac{N-s}{p-s},\mu_{s}}\right)^{\frac{1}{p-1}}\al\la^{\frac{N}{p^{*}}-\frac{N}{p_{2}}}.
\]
Hence by definition (\ref{eq: 2.2}), we obtain that
\[
\|w\|_{*,p_{1},p_{2},\la}\leq C\left(||a_{1}||_{\frac{N}{p}}+||a_{2}||_{\frac{N-s}{p-s},\mu_{s}}\right)^{\frac{1}{p-1}}\al.
\]
Since $\al>||v||_{*,p_{1},p_{2},\la}$ is arbitrary, we get (\ref{eq: 2.2.1-1}).
This finishes the proof.
\end{proof}
We also have the following result which will be used in the proof
of Proposition \ref{prop: main result in section 2}.
\begin{lem}
\label{lem: iteration} Let $w\in W_{0}^{1,p}(D)$, $w\ge0$, be
the solution of
\begin{equation}
\begin{cases}
-\Delta_{p}w=2\mu v^{p^{*}-1}+\frac{2v^{p^{*}(s)-1}}{|x|^{s}}+\frac{A}{|x|^{s}}, & \text{in }D,\\
w=0, & \text{on }\pa D,
\end{cases}\label{eq: 2.3.1}
\end{equation}
 where $v\ge0$ is a bounded function and $A\ge0$ is a constant.
Then for any $p_{1},p_{2}\in(p^{*}-1,\frac{N}{p}(p^{*}-1))$, $p_{2}<p^{*}<p_{1}$,
and for any $\la\ge1$, there exists a constant $C=C(p_{1},p_{2})>0$,
such that
\begin{equation}
\|w\|_{\ast,q_{1},q_{2},\la}\le C\|v\|_{*,p_{1},p_{2},\la}^{\frac{p^{*}-1}{p-1}}+C,\label{eq: 2.3.2}
\end{equation}
where $q_{1},q_{2}$ are given by
\begin{eqnarray*}
q_{1}=\frac{(p-1)N\hat{p}_{1}}{N-p\hat{p}_{1}} & \text{with} & \hat{p}_{1}=\frac{Np_{1}}{(p^{*}(s)-1)N+sp_{1}},
\end{eqnarray*}
and
\begin{eqnarray*}
q_{2}=\frac{(p-1)N\hat{p}_{2}}{N-p\hat{p}_{2}} & \text{with} & \hat{p}_{2}=\frac{p_{2}}{p^{*}-1}.
\end{eqnarray*}
\end{lem}
\begin{proof}
Let $\al>||v||_{*,p_{1},p_{2},\la}$ be an arbitrary constant. Then
by the definition of $||v||_{*,p_{1},p_{2},\la}$, there exist $v_{1},v_{2}$
such that $|v|\le v_{1}+v_{2}$ and (\ref{eq: 2.1}) holds with $u_{i}=v_{i}$,
$i=1,2$.

Let $w_{1}\in W_{0}^{1,p}(D)$, $w_{1}\ge0$, be the solution of equation
(\ref{eq: 2.3.1}) with $v=2v_{1}$. Let
\[
\hat{p}_{1}=\min\left\{ \frac{p_{1}}{p^{*}-1},\frac{Np_{1}}{\left(p^{*}(s)-1\right)N+sp_{1}}\right\} .
\]
By our assumptions on the parameters $N,p,s$ and $p_{1}$, we get
\[
\hat{p}_{1}=\frac{Np_{1}}{\left(p^{*}(s)-1\right)N+sp_{1}}\in\left(1,\frac{N}{p}\right)
\]
and $(p^{*}-1)\hat{p}_{1}\le p_{1}$, $(p^{*}(s)-1)\frac{(N-s)\hat{p}_{1}}{N-s\hat{p}_{1}}=\frac{(N-s)p_{1}}{N}$.
Thus applying Proposition \ref{prop: A.1} gives us
\begin{eqnarray*}
\|w_{1}\|_{\ast,q_{1}} & \le & C\left(\|v_{1}^{p^{*}-1}\|_{\hat{p}_{1}}+\|v_{1}^{p^{*}(s)-1}+A\|_{\frac{(N-s)\hat{p}_{1}}{N-s\hat{p}_{1}},\mu_{s}}\right)^{\frac{1}{p-1}}\\
 & \le & C\left(\|v_{1}\|_{p_{1}}^{p^{*}-1}+\|v_{1}\|_{\frac{(N-s)p_{1}}{N},\mu_{s}}^{p^{*}(s)-1}+1\right)^{\frac{1}{p-1}}\\
 & \le & C\al^{\frac{p^{*}-1}{p-1}}+C,
\end{eqnarray*}
where $q_{1}=(p-1)N\hat{p}_{1}/(N-p\hat{p}_{1})$.

Similarly, let $w_{2}\in W_{0}^{1,p}(D)$, $w_{2}\ge0$, be the solution
of equation
\[
\begin{cases}
-\Delta_{p}w=2\mu v^{p^{*}-1}+\frac{2v^{p^{*}(s)-1}}{|x|^{s}}, & \text{in }D,\\
w=0, & \text{on }\pa D.
\end{cases}
\]
Let
\[
\hat{p}_{2}=\min\left\{ \frac{p_{2}}{p^{*}-1},\frac{Np_{2}}{(p^{*}(s)-1)N+sp_{2}}\right\} .
\]
Then
\[
\hat{p}_{2}=\frac{p_{2}}{p^{*}-1}\in\left(1,\frac{N}{p}\right)
\]
and $\frac{(N-s)\hat{p}_{2}}{N-s\hat{p}_{2}}\le\frac{(N-s)p_{2}}{N}$,
$(p^{*}-1)\hat{p}_{2}=p_{2}$. Applying Proposition \ref{prop: A.1}
as above, we obtain that
\[
\|w_{2}\|_{\ast,q_{2}}\le\left(C\al^{\frac{p^{*}-1}{p-1}}+C\right)\la^{\frac{N}{p^{*}}-\frac{N}{q_{2}}},
\]
where $q_{2}=(p-1)N\hat{p}_{2}/(N-p\hat{p}_{2})$. To obtain the above
estimate, we used the equality
\[
\left(\frac{N}{p^{*}}-\frac{N}{p_{2}}\right)\frac{p^{*}-1}{p-1}=\frac{N}{p^{*}}-\frac{N}{q_{2}}.
\]

Let $\tilde{w}\in W_{0}^{1,p}(D)$, $\tilde{w}\ge0$ be the solution
of equation
\[
\begin{cases}
-\Delta_{p}w=2\mu\left((2v_{1})^{p^{*}-1}+(2v_{2})^{p^{*}-1}\right)+2\frac{(2v_{1})^{p^{*}(s)-1}+(2v_{2})^{p^{*}(s)-1}}{|x|^{s}}+\frac{A}{|x|^{s}}, & \text{in }D,\\
w=0, & \text{on }\pa D.
\end{cases}
\]
Estimating as above gives that
\[
\|\tilde{w}\|_{\ast,q_{2}}\le C\al^{\frac{p^{*}-1}{p-1}}+C,
\]
which implies that
\begin{eqnarray*}
\inf_{B_{r}(x)}\tilde{w}\le\left(\fint_{B_{r}(x)}\tilde{w}^{q_{2}}dy\right)^{\frac{1}{q_{2}}}\le C\al^{\frac{p^{*}-1}{p-1}}+C, &  & \forall\, x\in\Om.
\end{eqnarray*}

Note that $w\le\tilde{w}$ in $\Om$. Applying Proposition \ref{prop: C.1}
and arguing as that of Lemma \ref{lem: 2.2}, we prove Lemma \ref{lem: iteration}.
This completes the proof.
\end{proof}
Now define $u_{n}=0$ in $D\backslash\Om$. It is easy to see that
\[
\left|\mu|u|^{p^{*}-2-\ep}u+\frac{|u|^{p^{*}(s)-2-\ep}u}{|x|^{s}}+a(x)|u|^{p-2}u\right|\le2\mu|u|^{p^{*}-1}+\frac{2|u|^{p^{*}(s)-1}+A}{|x|^{s}}
\]
for sufficiently large constant $A>0$. Let $w_{n}\in W_{0}^{1,p}(D)$,
$w_{n}\ge0$, be the solution of equation
\begin{equation}
\begin{cases}
-\Delta_{p}w=2\mu|u_{n}|^{p^{*}-1}+\frac{2|u_{n}|^{p^{*}(s)-1}}{|x|^{s}}+\frac{A}{|x|^{s}}, & \text{in }D,\\
w=0, & \text{on }\pa D.
\end{cases}\label{eq: equation of w_n}
\end{equation}
Then by comparison principle,
\begin{eqnarray}
|u_{n}|\le w_{n} &  & \text{in }\Om.\label{eq: 2.15}
\end{eqnarray}
Moreover, since $\|u_{n}\|\le C$, it is easy to obtain from equation
(\ref{eq: equation of w_n}) that
\begin{equation}
\|w_{n}\|_{p^{*}}+\|w_{n}\|_{p^{*}(s),\mu_{s}}\le C\label{eq: 2.16}
\end{equation}
for some $C>0$ independent of $n$.

To prove Proposition \ref{prop: main result in section 2}, it is
enough to prove the estimate of Proposition \ref{prop: main result in section 2}
for $w_{n}$. We have the following result which shows that Proposition
\ref{prop: main result in section 2} holds for $w_{n}$ for some
$p_{1},p_{2}\in(p^{*}/p^{\prime},\wq)$, $p_{2}<p^{*}<p_{1}$.
\begin{lem}
 \label{lem: Initiation} There exist $p_{1},p_{2}\in(p^{*}/p^{\prime},\wq)$,
$p_{2}<p^{*}<p_{1}$, and constant $C=C(p_{1},p_{2})>0$, independent
of $n$, such that
\begin{equation}
\|w_{n}\|_{\ast,p_{1},p_{2},\la_{n}}\le C.\label{eq: initiation formula}
\end{equation}
\end{lem}
\begin{proof}
By Proposition \ref{prop:D.1}, $u_{n}$ can be decomposed as
\[
u_{n}=u_{0}+\sum_{j=1}^{k}\rho_{x_{n,j},\la_{n,j}}(U_{j})+\sum_{j=k+1}^{m}\rho_{0,\la_{n,j}}(U_{j})+\omega_{n}.
\]
Write $x_{n,j}=0$ for $j=k+1,\ldots,m$. In the following proof,
we denote
\begin{eqnarray*}
u_{n,0}=u_{0}, & {\displaystyle u_{n,1}=\sum_{j=1}^{m}\rho_{x_{n,j},\la_{n,j}}(U_{j})}, & \text{ and }u_{n,2}=\om_{n}.
\end{eqnarray*}
By (\ref{eq: 2.15}), we have
\[
2\mu|u_{n}|^{p^{*}-1}+\frac{2|u_{n}|^{p^{*}(s)-1}}{|x|^{s}}+\frac{A}{|x|^{s}}\le C\sum_{i=0}^{2}\left(|u_{n,i}|^{p^{*}-p}+\frac{|u_{n,i}|^{p^{*}(s)-p}}{|x|^{s}}\right)w_{n}^{p-1}+\frac{A}{|x|^{s}}.
\]

Let $\tilde{w}_{n}\in W_{0}^{1,p}(D)$, $\tilde{w}_{n}\ge0$, be the
solution of equation
\begin{equation}
\begin{cases}
-\Delta_{p}w={\displaystyle C\sum_{i=0}^{2}\left(|u_{n,i}|^{p^{*}-p}+\frac{|u_{n,i}|^{p^{*}(s)-p}}{|x|^{s}}\right)w_{n}^{p-1}+\frac{A}{|x|^{s}}}, & \text{in }D,\\
w=0, & \text{on }\pa D.
\end{cases}\label{eq: 2.18}
\end{equation}
 Comparison principle implies that
\begin{eqnarray*}
w_{n}\le\tilde{w}_{n}, &  & \text{in }D.
\end{eqnarray*}
By (\ref{eq: 2.16}), it is easy to derive that
\begin{equation}
\|\tilde{w}_{n}\|_{p^{*}}+\|\tilde{w}_{n}\|_{p^{*}(s),\mu_{s}}\le C.\label{eq: 2.19}
\end{equation}
Thus we have
\begin{eqnarray}
\inf_{B_{r}(x)}\tilde{w}_{n}\le C, &  & \forall\, x\in\Om.\label{eq: 2.20}
\end{eqnarray}

Now let $w_{i}\in W_{0}^{1,p}(D)$, $w_{i}\ge0$, $i=0,1,2$, be the
solution of equation
\[
\left\{ \begin{array}{ll}
{\displaystyle -\Delta_{p}w=C\left(|u_{n,i}|^{p^{*}-p}+\frac{|u_{n,i}|^{p^{*}(s)-p}}{|x|^{s}}\right)w_{n}^{p-1}+\frac{A\de_{i0}}{|x|^{s}},} & \text{in }D,\\
w=0, & \text{on }\pa D,
\end{array}\right.
\]
where $\de_{00}=1$ and $\de_{10}=\de_{20}=0$.

Then by Proposition \ref{prop: C.1} and (\ref{eq: 2.20}), we obtain
that
\begin{eqnarray}
\tilde{w}_{n}(x)\le C+Cw_{0}(x)+Cw_{1}(x)+Cw_{2}(x), &  & \forall\, x\in\Om.\label{eq: 2.21}
\end{eqnarray}
In the following we estimate $w_{i}$, $i=0,1,2$, term by term.

First we estimate $w_{0}$. We will use Proposition \ref{prop: A.1}
to estimate $w_{0}$. Since $0<s<p$, we can choose $q\ge1$ such
that
\[
\frac{s}{N}+\frac{p-1}{p^{*}}<\frac{1}{q}<\frac{p}{N}+\frac{p-1}{p^{*}}=\frac{p^{*}-1}{p^{*}}
\]
and that
\[
q<\frac{N}{p}.
\]
Then
\begin{eqnarray*}
\frac{(p-1)Nq}{N-pq}>p^{*} & \text{ and } & \frac{(p-1)(N-s)q}{N-sq}<p^{*}(s).
\end{eqnarray*}
Let $p_{1}=\frac{(p-1)Nq}{N-pq}$. Applying Proposition \ref{prop: A.1}
to $w_{0}$ gives us
\begin{eqnarray}
\|w_{0}\|_{\ast,p_{1}} & \le & C\left(\lv|u_{n,0}|^{p^{*}-p}w_{n}^{p-1}\rv_{q}+\lv|u_{n,0}|^{p^{*}(s)-p}w_{n}^{p-1}+A\rv_{\frac{(N-s)q}{N-sq},\mu_{s}}\right)^{\frac{1}{p-1}}\nonumber \\
 & \le & C\left(\lv w_{n}^{p-1}\rv_{q}+\lv w_{n}^{p-1}\rv_{\frac{(N-s)q}{N-sq},\mu_{s}}+1\right)^{\frac{1}{p-1}}\nonumber \\
 & \le & C\left(\lv w_{n}\rv_{(p-1)q}+\lv w_{n}\rv_{\frac{(p-1)(N-s)q}{N-sq},\mu_{s}}+1\right)\label{eq: 2.22}\\
 & \le & C\left(\|w_{n}\|_{p^{*}}+\|w_{n}\|_{p^{*}(s),\mu_{s}}+1\right)\nonumber \\
 & \le & C.\nonumber
\end{eqnarray}
Here in the second inequality we used the boundedness of $u_{n,0}=u_{0}$
and in the last inequality we used (\ref{eq: 2.16}). So this gives
estimate for $w_{0}$.

Next we use Corollary \ref{cor: A.3} to estimate $w_{1}$. We will
choose $p_{2}<p^{*}$, $p_{2}$ close to $p^{*}$ enough such that
\begin{equation}
\|w_{1}\|_{\ast,p_{2}}\le C\la_{n}^{\frac{N}{p^{*}}-\frac{N}{p_{2}}}.\label{eq: 2.23}
\end{equation}
Indeed, applying Corollary \ref{cor: A.3} to $w_{1}$ gives us that
\[
\|w_{1}\|_{\ast,p_{2}}\le C\left(\lv|u_{n,1}|^{p^{*}-p}\rv_{r_{1}}+\lv|u_{n,1}|^{p^{*}(s)-p}\rv_{r_{2},\mu_{s}}\right)^{\frac{1}{p-1}}\|w_{n}\|_{\ast,p^{*}},
\]
where  $r_{1},r_{2}$ are defined by
\begin{eqnarray*}
\frac{1}{r_{1}}=(p-1)\left(\frac{1}{p_{2}}-\frac{1}{p^{*}}\right)+\frac{p}{N} & \text{ and } & \frac{1}{r_{2}}=(p-1)\left(\frac{N}{(N-s)p_{2}}-\frac{1}{p^{*}(s)}\right)+\frac{p-s}{N-s}.
\end{eqnarray*}
By (\ref{eq: 2.16}), we have
\begin{equation}
\|w_{1}\|_{\ast,p_{2}}\le C\left(\lv|u_{n,1}|^{p^{*}-p}\rv_{r_{1}}+\lv|u_{n,1}|^{p^{*}(s)-p}\rv_{r_{2},\mu_{s}}\right)^{\frac{1}{p-1}}.\label{eq: 2.24}
\end{equation}
We only need to estimate $\lv|u_{n,1}|^{p^{*}-p}\rv_{r_{1}}$ and
$\lv|u_{n,1}|^{p^{*}(s)-p}\rv_{r_{2},\mu_{s}}$.

For all $1\le j\le m$, it is easy to see that
\[
\int_{\R^{N}}|\rho_{x_{n,j},\la_{n,j}}(U_{j})|^{\left(p^{*}-p\right)r_{1}}dy=\la_{n,j}^{pr_{1}-N}\int_{\R^{N}}|U_{j}|^{\left(p^{*}-p\right)r_{1}}dy.
\]
By Proposition \ref{prop: B.1}, for all $1\le j\le m$,
\begin{eqnarray*}
|U_{j}(y)|\le\frac{C}{1+|y|^{\frac{N-p}{p-1}}}, &  & \forall\, y\in\R^{N}.
\end{eqnarray*}
Since $\frac{N-p}{p-1}(p^{*}-p)r_{1}\to\frac{pN}{p-1}$ as $p_{2}\to p^{*}$,
we can choose $p_{2}$ close to $p^{*}$ enough such that $\frac{N-p}{p-1}(p^{*}-p)r_{1}>N$.
Then
\[
\int_{\R^{N}}|U_{j}|^{\left(p^{*}-p\right)r_{1}}dy<\wq.
\]
Thus for all $1\le j\le m$,
\[
\int_{\R^{N}}|\rho_{x_{n,j},\la_{n,j}}(U_{j})|^{\left(p^{*}-p\right)r_{1}}dy\le C\la_{n,j}^{pr_{1}-N}.
\]
Therefore
\begin{equation}
\begin{aligned}\lv|u_{n,1}|^{p^{*}-p}\rv_{r_{1}}^{\frac{1}{p-1}} & =\lv u_{n,1}\rv_{(p^{*}-p)r_{1}}^{\frac{p^{*}-p}{p-1}}\le C\sum_{j=1}^{m}\lv\rho_{x_{n,j},\la_{n,j}}(U_{j})\rv_{(p^{*}-p)r_{1}}^{\frac{p^{*}-p}{p-1}}\\
 & \le C\sum_{j=1}^{m}\la_{n,j}^{\frac{pr_{1}-N}{(p^{*}-p)r_{1}}\cdot\frac{p^{*}-p}{p-1}}\le C\la_{n}^{\frac{N}{p^{*}}-\frac{N}{p_{2}}}.
\end{aligned}
\label{eq: 2.25}
\end{equation}
We used the equality
\[
\frac{pr_{1}-N}{(p^{*}-p)r_{1}}\cdot\frac{p^{*}-p}{p-1}=\frac{N}{p^{*}}-\frac{N}{p_{2}}
\]
in the last inequality of (\ref{eq: 2.25}). This gives estimate for
$\lv|u_{n,1}|^{p^{*}-p}\rv_{r_{1}}$.

We can also choose $p_{2}$ close to $p^{*}$ enough such that for
all $1\le j\le m$,
\[
\int_{\R^{N}}|\rho_{x_{n,j},\la_{n,j}}(U_{j})|^{\left(p^{*}(s)-p\right)r_{2}}d\mu_{s}\le C\la_{n,j}^{(p-s)r_{2}-N+s}.
\]
Indeed, we have
\[
\int_{\R^{N}}|\rho_{x_{n,j},\la_{n,j}}(U_{j})|^{\left(p^{*}(s)-p\right)r_{2}}d\mu_{s}=\la_{n,j}^{(p-s)r_{2}-N+s}\int_{\R^{N}}\frac{|U_{j}(y)|^{\left(p^{*}(s)-p\right)r_{2}}}{|y+\la_{n,j}x_{n,j}|^{s}}dy.
\]
Write $y_{n,j}=-\la_{n,j}x_{n,j}$. Let
\begin{eqnarray*}
I_{1}=\int_{B_{1}(y_{n,j})}\frac{|U_{j}(y)|^{\left(p^{*}(s)-p\right)r_{2}}}{|y-y_{n,j}|^{s}}dy, & \text{ and } & I_{2}=\int_{\R^{N}\backslash B_{1}(y_{n,j})}\frac{|U_{j}(y)|^{\left(p^{*}(s)-p\right)r_{2}}}{|y-y_{n,j}|^{s}}dy.
\end{eqnarray*}
Since $U_{j}$ is bounded and $0<s<N$, we have
\[
I_{1}\le C.
\]
Let $\de>0$ be a number to be determined. By H\"older's inequality,
we have
\begin{eqnarray*}
I_{2} & \le & \left(\int_{\R^{N}\backslash B_{1}(y_{n,j})}\frac{1}{|y-y_{n,j}|^{N+\de}}dy\right)^{\frac{s}{N+\de}}\left(\int_{\R^{N}\backslash B_{1}(y_{n,j})}|U_{j}(y)|^{\frac{\left(p^{*}(s)-p\right)r_{2}(N+\de)}{N+\de-s}}dy\right)^{\frac{N+\de-s}{N+\de}}\\
 & \le & C_{\de}\left(\int_{\R^{N}}|U_{j}(y)|^{\frac{\left(p^{*}(s)-p\right)r_{2}(N+\de)}{N+\de-s}}dy\right)^{\frac{N+\de-s}{N+\de}}.
\end{eqnarray*}
Since
\begin{eqnarray*}
\frac{N-p}{p-1}\frac{\left(p^{*}(s)-p\right)r_{2}(N+\de)}{N+\de-s}\to\frac{p(N-s)(N+\de)}{(p-1)(N+\de-s)} &  & \text{as }p_{2}\to p^{*},
\end{eqnarray*}
and
\begin{eqnarray*}
\frac{p(N-s)(N+\de)}{(p-1)(N+\de-s)}>N &  & \text{for }\de>0\text{ small enough},
\end{eqnarray*}
we can $p_{2}$ close to $p^{*}$ enough and $\de>0$ small enough
such that $\frac{N-p}{p-1}\frac{\left(p^{*}(s)-p\right)r_{2}(N+\de)}{N+\de-s}>N$.
Then
\[
\int_{\R^{N}}|U_{j}(y)|^{\frac{\left(p^{*}(s)-p\right)r_{2}(N+\de)}{N+\de-s}}dy<\wq.
\]
Then we obtain that
\[
I_{2}\le C.
\]
Combining the estimates of $I_{1}$ and $I_{2}$ we obtain that
\[
\int_{\R^{N}}|\rho_{x_{n,j},\la_{n,j}}(U_{j})|^{\left(p^{*}(s)-p\right)r_{2}}d\mu_{s}\le C\la_{n,j}^{(p-s)r_{2}-N+s}.
\]
Hence we have
\begin{equation}
\begin{aligned}\lv|u_{n,1}|^{p^{*}(s)-p}\rv_{r_{2},\mu_{s}}^{\frac{1}{p-1}} & =\lv u_{n,1}\rv_{(p^{*}(s)-p)r_{2},\mu_{s}}^{\frac{p^{*}(s)-p}{p-1}}\\
 & \:\le C\sum_{j=1}^{m}\la_{n,j}^{\frac{(p-s)r_{2}-N+s}{\left(p^{*}(s)-p\right)r_{2}}\cdot\frac{p^{*}(s)-p}{p-1}}\\
 & \le C\la_{n}^{\frac{N}{p^{*}}-\frac{N}{p_{2}}}.
\end{aligned}
\label{eq: 2.26}
\end{equation}
In the above inequality we used the equality
\[
\frac{(p-s)r_{2}-N+s}{\left(p^{*}(s)-p\right)r_{2}}\cdot\frac{p^{*}(s)-p}{p-1}=\frac{N}{p^{*}}-\frac{N}{p_{2}}.
\]
Combining (\ref{eq: 2.24})-(\ref{eq: 2.26}) gives (\ref{eq: 2.23}).

Finally we use Lemma \ref{lem: 2.2} to estimate $w_{2}$. By Lemma
\ref{lem: 2.2}, we have
\begin{equation}
\begin{aligned}\|w_{2}\|_{\ast,p_{1},p_{2},\la_{n}} & \le C\left(\||u_{n,2}|^{p^{*}-p}\|_{\frac{N}{p}}+\||u_{n,2}|^{p^{*}(s)-p}\|_{\frac{N-s}{p-s},\mu_{s}}\right)^{\frac{1}{p-1}}\|w_{n}\|_{\ast,p_{1},p_{2},\la_{n}}\\
 & \le\frac{1}{2C^{\prime}}\|w_{n}\|_{\ast,p_{1},p_{2},\la_{n}},
\end{aligned}
\label{eq: 2.27}
\end{equation}
since $\om_{n}\to0$ in $W_{0}^{1,p}(\Om)$, where the constant $C^{\prime}$
is given by (\ref{eq: 2.21}).

Now combining (\ref{eq: 2.16}), (\ref{eq: 2.21})-(\ref{eq: 2.23})
and (\ref{eq: 2.27}), we obtain that
\begin{eqnarray*}
\|\tilde{w}_{n}\|_{\ast,p_{1},p_{2},\la_{n}} & \le & C+C\|w_{0}\|_{\ast,p_{1},p_{2},\la_{n}}+C\|w_{1}\|_{\ast,p_{1},p_{2},\la_{n}}+C\|w_{2}\|_{\ast,p_{1},p_{2},\la_{n}}\\
 & \le & C+C\|w_{0}\|_{\ast,p_{1}}+C\|w_{1}\|_{\ast,p_{1}}+\frac{1}{2}\|w_{n}\|_{\ast,p_{1},p_{2},\la_{n}}\\
 & \le & C+\frac{1}{2}\|\tilde{w}_{n}\|_{\ast,p_{1},p_{2},\la_{n}},
\end{eqnarray*}
which completes the proof.
\end{proof}
Now we can prove Proposition \ref{prop: main result in section 2}.

\begin{proof}[Proof of Proposition \ref{prop: main result in section 2}]Since
$w_{n}$ is a solution to equation (\ref{eq: equation of w_n}), we
can use Lemma \ref{lem: 2.2} and Lemma \ref{lem: iteration} to prove
Proposition \ref{prop: main result in section 2}. See details in
e.g. \cite{CaoPengYan2012}. This finishes the proof of Proposition
\ref{prop: main result in section 2}. \end{proof}

\section{Estimates on safe regions}

Since the number of the bubbles of $u_{n}$ is finite, by Proposition
\ref{prop:D.1} we can always find a constant $\bar{C}>0,$ independent
of $n$, such that the region
\[
\mathcal{A}_{n}^{1}=\Bigl(B_{(\bar{C}+5)\la_{n}^{-\frac{1}{p}}}(x_{n})\backslash B_{\bar{C}\la_{n}^{-\frac{1}{p}}}(x_{n})\Bigl)\cap\,\Omega
\]
does not contain any concentration point of $u_{n}$ for any $n$.
We call this region a safe region for $u_{n}.$

Let
\[
\mathcal{A}_{n}^{2}=\Bigl(B_{(\bar{C}+4)\la_{n}^{-\frac{1}{p}}}(x_{n})\backslash B_{(\bar{C}+1)\la_{n}^{-\frac{1}{p}}}(x_{n})\Bigl)\cap\,\Omega.
\]
In this section, we prove the following result.
\begin{prop}
\label{prop: 3.1} Let $u_{n}$ be a solution of equation (\ref{1.3})
with $\ep=\ep_{n}\to0$, satisfying $||u_{n}||\le C$ for some positive
constant $C$ independent of $n$.  Then for any constant $q\geq p$,
there is a constant $C>0$ independent of $n$, such that
\[
\int_{\mathcal{A}_{n}^{2}}|u_{n}|^{q}dx\leq C\la_{n}^{-\frac{N}{p}}.
\]

\end{prop}
In order to prove Proposition \ref{prop: 3.1}, we need the following
lemma.
\begin{lem}
\label{lem: 3.2} Let $D$ be a bounded domain with $\Om\subset\subset D$
and $w_{n}$ the solution of equation (\ref{eq: equation of w_n}).
Then there exist a number $\ensuremath{\gamma>p-1}$ and a constant
$C>0$ independent of $n,$ such that
\begin{eqnarray*}
\Big(\frac{1}{r^{N}}\int_{B_{r}(y)\cap\Om}w_{n}^{\gamma}dx\Big)^{\frac{1}{\gamma}}\leq C, &  & \forall y\in\Om,
\end{eqnarray*}
for all $r\ge\bar{C}\la_{n}^{-\frac{1}{p}}$.\end{lem}
\begin{proof}
We will combine Proposition \ref{prop: main result in section 2}
and Proposition \ref{prop: C.2} to prove Lemma \ref{lem: 3.2}. Since
$w_{n}$ is the solution of equation (\ref{eq: equation of w_n}),
applying proposition \ref{prop: C.2} gives us a number $\ga\in(p-1,(p-1)N/(N-p+1))$
and a constant $C=C(N,p,\ga)$ such that
\[
\begin{aligned}\Big(\frac{1}{r^{N}}\int_{B_{r}(y)\cap\Om}w_{n}^{\gamma}dx\Big)^{\frac{1}{\gamma}} & \le C+C\int_{r}^{R}\left(\frac{1}{t^{N-p}}\int_{B_{t}(y)}\left(2\mu|u_{n}|^{p^{*}-1}+\frac{2|u_{n}|^{p^{*}(s)-1}}{|x|^{s}}+\frac{A}{|x|^{s}}\right)dx\right)^{\frac{1}{p-1}}\frac{dt}{t}\\
 & \le C+C\int_{r}^{R}\left(\frac{1}{t^{N-p}}\int_{B_{t}(y)}\left(|u_{n}|^{p^{*}-1}+\frac{|u_{n}|^{p^{*}(s)-1}}{|x|^{s}}\right)dx\right)^{\frac{1}{p-1}}\frac{dt}{t},
\end{aligned}
\]
for all $0<r<R$, where $R=\dist(\Om,\pa D)$. Let
\[
I_{1}=\int_{r}^{R}\left(\frac{1}{t^{N-p}}\int_{B_{t}(y)}|u_{n}|^{p^{*}-1}dx\right)^{\frac{1}{p-1}}\frac{dt}{t}
\]
and
\[
I_{2}=\int_{r}^{R}\left(\frac{1}{t^{N-p}}\int_{B_{t}(y)}\frac{|u_{n}|^{p^{*}(s)-1}}{|x|^{s}}dx\right)^{\frac{1}{p-1}}\frac{dt}{t}
\]
such that
\begin{equation}
\Big(\frac{1}{r^{N}}\int_{B_{r}(y)\cap\Om}w_{n}^{\gamma}dx\Big)^{\frac{1}{\gamma}}\le C+CI_{1}+CI_{2}.\label{eq: 3.2.1}
\end{equation}
We now estimate $I_{1}$ and $I_{2}$ for $r\ge\bar{C}\la_{n}^{-1/p}$.

By Proposition \ref{prop: main result in section 2}, $\|u_{n}\|_{\ast,p_{1},p_{2},\la}\le C$
for any $p_{1},p_{2}\in(p^{*}/p^{\prime},\wq)$, $p_{2}<p^{*}<p_{1}$.

Let $p_{1}>p^{*}$ be a number to be determined and $p_{2}=p^{*}-1$.
There exist $u_{n,1}$, $u_{n,2}$ with $|u_{n}|\le u_{n,1}+u_{n,2}$
such that $\|u_{n,1}\|_{\ast,p_{1}}\le C$ and $\|u_{n,2}\|_{\ast,p_{2}}\le C\la_{n}^{\frac{N}{p^{*}}-\frac{N}{p_{2}}}$.
Then
\[
\int_{B_{t}(y)}|u_{n,1}|^{p^{*}-1}dx\le C\left(\int_{B_{t}(y)}|u_{n,1}|^{p_{1}}dx\right)^{\frac{p^{*}-1}{p_{1}}}|B_{t}(y)|^{1-\frac{p^{*}-1}{p_{1}}}\le Ct^{\left(1-\frac{p^{*}-1}{p_{1}}\right)N},
\]
and
\[
\int_{B_{t}(y)}|u_{n,2}|^{p^{*}-1}dx=\int_{B_{t}(y)}|u_{n,2}|^{p_{2}}dx\le C\la_{n}^{\left(\frac{N}{p^{*}}-\frac{N}{p_{2}}\right)p_{2}}=C\la_{n}^{\frac{p-N}{p}}.
\]
Thus
\begin{eqnarray*}
\int_{B_{t}(y)}|u_{n}|^{p^{*}-1}dx & \le & C\int_{B_{t}(y)}|u_{n,1}|^{p^{*}-1}dx+C\int_{B_{t}(y)}|u_{n,2}|^{p^{*}-1}dx\\
 & \le & Ct^{\left(1-\frac{p^{*}-1}{p_{1}}\right)N}+C\la_{n}^{\frac{p-N}{p}}.
\end{eqnarray*}
Since $\frac{N}{p-1}\left(1-\frac{p^{*}-1}{p_{1}}\right)+\frac{p-N}{p-1}\to\frac{p}{p-1}$
as $p_{1}\to\wq$, we can choose $p_{1}>p^{*}$ large enough such
that $\frac{N}{p-1}\left(1-\frac{p^{*}-1}{p_{1}}\right)+\frac{p-N}{p-1}>0$.
Then
\[
\int_{0}^{R}t^{\left(1-\frac{p^{*}-1}{p_{1}}\right)\frac{N}{p-1}+\frac{p-N}{p-1}}\frac{dt}{t}<C.
\]
Note also that for $r\ge\bar{C}\la_{n}^{-1/p}$, we have
\[
\int_{r}^{\wq}t^{\frac{p-N}{p-1}}\frac{dt}{t}\le C\la_{n}^{\frac{N-p}{p(p-1)}}.
\]
 Therefore
\begin{eqnarray}
I_{1} & \le & \int_{r}^{R}\left(Ct^{\left(1-\frac{p^{*}-1}{p_{1}}\right)N}+C\la_{n}^{\frac{p-N}{p}}\right)^{\frac{1}{p-1}}t^{\frac{p-N}{p-1}}\frac{dt}{t}\nonumber \\
 & \le & C\int_{0}^{R}t^{\frac{N}{p-1}\left(1-\frac{p^{*}-1}{p_{1}}\right)+\frac{p-N}{p-1}}\frac{dt}{t}+C\la_{n}^{\frac{p-N}{p(p-1)}}\int_{r}^{\wq}t^{\frac{p-N}{p-1}}\frac{dt}{t}\label{eq: 3.2.2}\\
 & \le & C.\nonumber
\end{eqnarray}
This gives estimate for $I_{1}$.

Next we estimate $I_{2}$. Let $p_{1}>p^{*}$ to be determined and
$p_{2}=N\left(p^{*}(s)-1\right)/(N-s)$. There exist $\bar{u}_{n,1}$,
$\bar{u}_{n,2}$ with $|u_{n}|\le\bar{u}_{n,1}+\bar{u}_{n,2}$ such
that $\|\bar{u}_{n,1}\|_{\ast,p_{1}}\le C$ and $\|\bar{u}_{n,2}\|_{\ast,p_{2}}\le C\la_{n}^{\frac{N}{p^{*}}-\frac{N}{p_{2}}}$.
Then
\begin{eqnarray*}
\int_{B_{t}(y)}|\bar{u}_{n,1}|^{p^{*}(s)-1}d\mu_{s} & \le & \left(\int_{B_{t}(y)}|\bar{u}_{n,1}|^{\frac{N-s}{N}p_{1}}d\mu_{s}\right)^{\frac{\left(p^{*}(s)-1\right)N}{(N-s)p_{1}}}\left(\int_{B_{t}(y)}d\mu_{s}\right)^{1-\frac{\left(p^{*}(s)-1\right)N}{(N-s)p_{1}}}\\
 & \le & Ct^{N-s-\frac{\left(p^{*}(s)-1\right)N}{p_{1}}},
\end{eqnarray*}
and
\[
\int_{B_{t}(y)}|\bar{u}_{n,2}|^{p^{*}(s)-1}d\mu_{s}=\int_{B_{t}(y)}|\bar{u}_{n,2}|^{\frac{N-s}{N}p_{2}}d\mu_{s}\le C\la_{n}^{\frac{p-N}{p}}.
\]
Arguing as above yields that
\begin{equation}
I_{2}\le C,\label{eq: 3.2.3}
\end{equation}
if we choose $p_{1}$ large enough. This gives estimate for $I_{2}$.

Combining (\ref{eq: 3.2.1})-(\ref{eq: 3.2.3}), we complete the proof
of Lemma \ref{lem: 3.2}.
\end{proof}
Now we can prove Proposition \ref{prop: 3.1}.

\begin{proof}[Proof of Proposition \ref{prop: 3.1}] Let $\ga>p-1$
be as in Lemma \ref{lem: 3.2}. Since $|u_{n}|\le w_{n}$, we have
\begin{eqnarray}
\int_{B_{\la_{n}^{-1/p}}(y)}|u_{n}|^{\ga}dx\leq C\la_{n}^{-\frac{N}{p}}, &  & \forall y\in{\cal A}_{n}^{2}.\label{eq: 3.1.0}
\end{eqnarray}

Let $v_{n}(x)=u_{n}(\la_{n}^{-\frac{1}{p}}x)$, $x\in\Om_{n}=\{x;\la_{n}^{-\frac{1}{p}}x\in\Om\}$.
Then $v_{n}$ is a solution to equation
\[
\begin{cases}
{\displaystyle -\Delta_{p}v_{n}=\la_{n}^{-1}\left(\mu|v_{n}|^{p^{*}-p-\ep_{n}}+\frac{\la_{n}^{\frac{s}{p}}|v_{n}|^{p^{*}(s)-p-\ep_{n}}}{|x|^{s}}+a(\la_{n}^{-\frac{1}{p}}x)\right)|v_{n}|^{p-2}v_{n}}, & x\in\Om_{n},\\
v_{n}=0, & \text{on }\pa\Om_{n}.
\end{cases}
\]

Let $\ensuremath{z=\la_{n}^{\frac{1}{p}}y}$, $y\in{\cal A}_{n}^{2}$.
Since $B_{\la_{n}^{-1/p}}(y)$ does not contain any concentration
point of $u_{n}$, we can deduce that \begin{eqnarray*} \int_{B_{1}(z)}\big|\lambda_{n}^{-1}\big(\mu|v_{n}|^{p^{*}-p-\epsilon_{n}} +a(\lambda_{n}^{-\frac{1}{p}}x)\big)\big|^{\frac{N}{p}}dx &\leq& C\int_{B_{1}(z)} |\lambda_{n}^{-1}(|v_{n}|^{p^{*}-p}+1)|^{\frac{N}{p}}dx \\ &\leq& C\int_{B_{\lambda_{n}^{-\frac{1}{p}}}(y)} |u_{n}|^{p^{*}}dx+C\lambda_{n}^{-\frac{N}{p}} \rightarrow 0, \end{eqnarray*}and
that \begin{eqnarray*} \int_{B_{1}(z)}\frac{|\lambda_{n}^{-1}\lambda_{n}^{\frac{s}{p}} |v_{n}|^{p^{*}(s)-p-\epsilon_{n}}|^{\frac{N-s}{p-s}}}{|x|^{s}}dx  &\leq&C\int_{B_{1}(z)}\frac{|\lambda_{n}^{\frac{s-p}{p}} (|v_{n}|^{p^{*}(s)-p}+1)|^{\frac{N-s}{p-s}}}{|x|^{s}}dx \\ &\leq&C\int_{B_{\lambda_{n}^{-\frac{1}{p}}}(y)} \frac{|u_{n}|^{p^{*}(s)}}{|x|^{s}}dx+C\lambda_{n}^{-\frac{N-s}{p}} \rightarrow 0, \end{eqnarray*} as
$n\to\wq$.

Thus for any $q>p^{*}$, we obtain by Lemma \ref{lem: A.4} and (\ref{eq: 3.1.0})
that,
\[
\|v_{n}\|_{q,B_{1/2}(z)}\le C\left(\int_{B_{1}(z)}|v_{n}|^{\ga}dx\right)^{\frac{1}{\ga}}=C\left(\fint_{B_{\la_{n}^{-1/p}}(y)}|u_{n}|^{\ga}dx\right)^{\frac{1}{\ga}}\le C.
\]
Equivalently, we arrive at
\begin{eqnarray*}
\int_{B_{\frac{1}{2}\la_{n}^{-1/p}}(y)}|u_{n}|^{q}dx\leq C\la_{n}^{-\frac{N}{p}}, &  & \forall y\in{\cal A}_{n}^{2}.
\end{eqnarray*}
Now a simple covering argument proves Proposition \ref{prop: 3.1}
in the case when $q>p^{*}$.

If $p\le q\le p^{*}<2p^{*}$, we apply H\"older's inequality to obtain
that
\[
\left(\fint_{{\cal A}_{n}^{2}}|u_{n}|^{q}dx\right)^{\frac{1}{q}}\le\left(\fint_{{\cal A}_{n}^{2}}|u_{n}|^{2p^{*}}dx\right)^{\frac{1}{2p^{*}}}\le C.
\]
We complete the proof of Proposition \ref{prop: 3.1}.\end{proof}

Let
\[
\mathcal{A}_{n}^{3}=\Big(B_{(\bar{C}+3)\Lambda_{n}^{-\frac{1}{p}}}(x_{n})\backslash B_{(\bar{C}+2)\Lambda_{n}^{-\frac{1}{p}}}(x_{n})\Big)\cap\Omega.
\]
In the end of this section, we prove the following estimate for $u_{n}$.
\begin{prop}
\label{prop: 3.3 gradient estimate} We have
\begin{equation}
\int_{\mathcal{A}_{n}^{3}}|\nabla u_{n}|^{p}dx\leq C\int_{\mathcal{A}_{n}^{2}}\Big(|u_{n}|^{p^{*}}+\frac{|u_{n}|^{p^{*}(s)}}{|x|^{s}}+1\Big)dx+C\la_{n}\int_{\mathcal{A}_{n}^{2}}|u_{n}|^{p}dx.\label{eq: 3.3.1}
\end{equation}
In particular, we have
\begin{equation}
\int_{\mathcal{A}_{n}^{3}}|\nabla u_{n}|^{p}dx\leq C\la_{n}^{\frac{p-N}{p}}.\label{eq:3.3.2}
\end{equation}
\end{prop}
\begin{proof}
Let $\phi\in C_{0}^{\infty}(\mathcal{A}_{n}^{2})$ be a cut-off function
with $\phi=1$ in $\mathcal{A}_{n}^{3}$, $0\leq\phi\leq1$ and $|\nabla\phi|\leq C\la_{n}^{\frac{1}{p}}$.
Multiplying the equation of $u_{n}$ by $\phi^{p}u_{n}$ yields that
\[
\int_{{\cal A}_{n}^{2}}|\na u_{n}|^{p-2}\na u_{n}\cdot\na(\phi^{p}u_{n})dx=\int_{{\cal A}_{n}^{2}}\left(\mu|u_{n}|^{p^{*}-2-\ep_{n}}u_{n}+\frac{|u_{n}|^{p^{*}(s)-2-\ep_{n}}u_{n}}{|x|^{s}}+a|u_{n}|^{p-2}u_{n}\right)\phi^{p}u_{n}dx.
\]
 It is easy to derive (\ref{eq: 3.3.1}) from the equality above.

Let $q>p^{*}(s)$. By Proposition \ref{prop: 3.1}, we have
\begin{eqnarray}
\int_{{\cal A}_{n}^{2}}\frac{\phi^{p}|u_{n}|^{p^{*}(s)}}{|x|^{s}}dx & \le & \left(\int_{{\cal A}_{n}^{2}}\phi^{p}|u_{n}|^{q}dx\right)^{\frac{p^{*}(s)}{q}}\left(\int_{{\cal A}_{n}^{2}}\phi^{p}|x|^{-\frac{sq}{q-p^{*}(s)}}dx\right)^{\frac{q-p^{*}(s)}{q}}\nonumber \\
 & \le & C\la_{n}^{\frac{p^{*}(s)N}{pq}}\la_{n}^{-\frac{1}{p}\left(N-\frac{sq}{q-p^{*}(s)}\right)\left(\frac{q-p^{*}(s)}{q}\right)}\label{eq: 3.3.3}\\
 & = & C\la_{n}^{\frac{s-N}{p}}.\nonumber
\end{eqnarray}
Now from (\ref{eq: 3.3.3}), (\ref{eq: 3.3.1}) and Proposition \ref{prop: 3.1},
we obtain that
\[
\int_{\mathcal{A}_{n}^{3}}|\nabla u_{n}|^{p}dx\le C\la_{n}^{-\frac{N}{p}}+C\la_{n}^{\frac{p-N}{p}}+C\la_{n}^{\frac{s-N}{p}}\le C\la_{n}^{\frac{p-N}{p}}.
\]
This proves (\ref{eq:3.3.2}). We finish the proof.
\end{proof}

\section{Proof of main results}

In this section we prove Theorem \ref{thm: main result} and Theorem
\ref{thm: infinitely many solutions}.

For simplicity, write $p_{n}=p^{*}-\ep_{n}$ and $p_{n}(s)=p^{*}(s)-\ep_{n}$.
Choose $t_{n}\in[\bar{C}+2,\bar{C}+3]$ such that
\begin{equation}
\begin{array}{ll}
 & {\displaystyle \int_{\partial B_{t_{n}\lambda_{n}^{-\frac{1}{p}}}(x_{n})}\Bigl(\mu|u_{n}|^{p_{n}}+|u_{n}|^{p}+\lambda_{n}^{-1}|\nabla u_{n}|^{p}+\lambda_{n}^{-\frac{s}{p}}\frac{|u_{n}|^{p_{n}(s)}}{|x|^{s}}\Bigl)d\sigma}\vspace{2mm}\\
\leq & {\displaystyle C\lambda_{n}^{\frac{1}{p}}\int_{\mathcal{A}_{n}^{3}}\Bigl(\mu|u_{n}|^{p_{n}}+|u_{n}|^{p}+\lambda_{n}^{-1}|\nabla u_{n}|^{p}+\lambda_{n}^{-\frac{s}{p}}\frac{|u_{n}|^{p_{n}(s)}}{|x|^{s}}\Bigl)dx.}
\end{array}\label{eq: 4.1}
\end{equation}
By Proposition \ref{prop: 3.1}, (\ref{eq:3.3.2}) and (\ref{eq: 3.3.3}),
we obtain that
\begin{equation}
\int_{\partial B_{t_{n}\lambda_{n}^{-\frac{1}{p}}}(x_{n})}\left(\mu|u_{n}|^{p_{n}}+|u_{n}|^{p}+\lambda_{n}^{-1}|\nabla u_{n}|^{p}+\lambda_{n}^{-\frac{s}{p}}\frac{|u_{n}|^{p_{n}(s)}}{|x|^{s}}\right)d\sigma\le C\la_{n}^{\frac{1-N}{p}}.\label{eq: 4.2}
\end{equation}

We also have the following Pohozaev identity for $u_{n}$ on $B_{n}=B_{t_{n}\lambda_{n}^{-\frac{1}{p}}}(x_{n})\cap\Om$
\begin{eqnarray*}
 &  & \left(\frac{N}{p_{n}}-\frac{N-p}{p}\right)\mu\int_{B_{n}}|u_{n}|^{p_{n}}dx+\int_{B_{n}}\Big[a(x)-\frac{1}{p}\nabla a(x)\cdot(x-x_{0})\Big]|u_{n}|^{p}dx\vspace{2mm}\\
 &  & +\left(\frac{N-s}{p_{n}(s)}-\frac{N-p}{p}\right)\int_{B_{n}}\frac{|u_{n}|^{p_{n}(s)}}{|x|^{s}}dx+\frac{s}{p_{n}(s)}\int_{B_{n}}\frac{|u_{n}|^{p_{n}(s)}}{|x|^{2+s}}(x_{0}\cdot x)dx\vspace{2mm}\\
 & = & \frac{N-p}{p}\int_{\partial B_{n}}|\nabla u_{n}|^{p-2}\frac{\partial u_{n}}{\partial\nu}u_{n}d\sigma+\int_{\partial B_{n}}|\nabla u_{n}|^{p-2}\nabla u_{n}\cdot(x-x_{0})\frac{\partial u_{n}}{\partial\nu}d\sigma\vspace{2mm}\\
 &  & \,\,\,-\frac{1}{p}\int_{\partial B_{n}}|\nabla u_{n}|^{p}(x-x_{0})\cdot\nu d\sigma\vspace{2mm}\\
 &  & \,\,\,+\int_{\partial B_{n}}(x-x_{0})\cdot\nu\Big[\frac{1}{p_{n}}|u_{n}|^{p_{n}}+\frac{1}{p_{n}(s)}\frac{|u_{n}|^{p_{n}(s)}}{|x|^{s}}+\frac{1}{p}a(x)|u_{n}|^{p}\Big]d\sigma,
\end{eqnarray*}
where $\nu$ is the outward unit normal to $\partial B_{n}$ and $x_{0}\in\R^{N}$.
Since $p_{n}<p^{*}$ and $p_{n}(s)<p^{*}(s)$, we have the following
inequality from above
\begin{equation}
\begin{aligned} & \int_{B_{n}}\Big[a(x)-\frac{1}{p}\nabla a(x)\cdot(x-x_{0})\Big]|u_{n}|^{p}dx+\frac{s}{p_{n}(s)}\int_{B_{n}}\frac{|u_{n}|^{p_{n}(s)}}{|x|^{2+s}}(x_{0}\cdot x)dx\vspace{4mm}\\
\le\; & \frac{N-p}{p}{\displaystyle \int_{\partial B_{n}}|\nabla u_{n}|^{p-2}\frac{\partial u_{n}}{\partial\nu}u_{n}d\sigma}+\int_{\partial B_{n}}|\nabla u_{n}|^{p-2}\nabla u_{n}\cdot(x-x_{0})\frac{\partial u_{n}}{\partial\nu}d\sigma\vspace{4mm}\\
 & \,\,\,-\frac{1}{p}\int_{\partial B_{n}}|\nabla u_{n}|^{p}(x-x_{0})\cdot\nu d\sigma\vspace{4mm}\\
 & \,\,\,+\int_{\partial B_{n}}(x-x_{0})\cdot\nu\Big[\frac{1}{p_{n}}|u_{n}|^{p_{n}}+\frac{1}{p_{n}(s)}\frac{|u_{n}|^{p_{n}(s)}}{|x|^{s}}+\frac{1}{p}a(x)|u_{n}|^{p}\Big]d\sigma.
\end{aligned}
\label{eq: 4.4}
\end{equation}
 Now we can prove Theorem \ref{thm: main result}.

\begin{proof}[Proof of Theorem \ref{thm: main result}] Since $\{x_{n}\}\subset\Om$
is a bounded sequence, we may assume that $x_{n}\to x^{*}\in\bar{\Om}$
as $n\to\wq$. We have two cases:

\noindent\emph{Case 1. $x^{*}=0$;}

\noindent\emph{Case 2. $x^{*}\neq0$. }

In Case 1, choose $x_{0}=0$ in (\ref{eq: 4.4}). Then we obtain that
\begin{equation}
\begin{aligned} & \int_{B_{n}}\Big[a(x)-\frac{1}{p}\nabla a(x)\cdot x\Big]|u_{n}|^{p}dx\vspace{4mm}\\
\le\: & \frac{N-p}{p}{\displaystyle \int_{\partial B_{n}}|\nabla u_{n}|^{p-2}\frac{\partial u_{n}}{\partial\nu}u_{n}d\sigma}+\int_{\partial B_{n}}|\nabla u_{n}|^{p-2}\nabla u_{n}\cdot x\frac{\partial u_{n}}{\partial\nu}d\sigma\vspace{4mm}\\
 & \,\,\,-\frac{1}{p}\int_{\partial B_{n}}|\nabla u_{n}|^{p}x\cdot\nu d\sigma\vspace{4mm}\\
 & \,\,\,+\int_{\partial B_{n}}x\cdot\nu\Big[\frac{1}{p_{n}}|u_{n}|^{p_{n}}+\frac{1}{p_{n}(s)}\frac{|u_{n}|^{p_{n}(s)}}{|x|^{s}}+\frac{1}{p}a(x)|u_{n}|^{p}\Big]d\sigma.
\end{aligned}
\label{eq: 4.5}
\end{equation}

Decompose $\pa B_{n}$ by $\pa B_{n}=\pa_{i}B_{n}\cup\pa_{e}B_{n}$,
where $\pa_{i}B_{n}=\pa B_{n}\cap\Om$ and $\pa_{e}B_{n}=\pa B_{n}\cap\pa\Om$.

Consider the case $0\in\pa\Om$ first. Observe that $u_{n}=0$ on
$\pa\Om$. Thus (\ref{eq: 4.5}) implies that
\begin{equation}
\begin{aligned}L_{1}:= & \int_{B_{n}}\Big[a(x)-\frac{1}{p}\nabla a(x)\cdot x\Big]|u_{n}|^{p}dx-\left(1-\frac{1}{p}\right)\int_{\partial_{e}B_{n}}|\nabla u_{n}|^{p}x\cdot\nu d\sigma\vspace{4mm}\\
\le & \frac{N-p}{p}{\displaystyle \int_{\pa_{i}B_{n}}|\nabla u_{n}|^{p-2}\frac{\partial u_{n}}{\partial\nu}u_{n}d\sigma}+\int_{\pa_{i}B_{n}}|\nabla u_{n}|^{p-2}\nabla u_{n}\cdot x\frac{\partial u_{n}}{\partial\nu}d\sigma\vspace{4mm}\\
 & \,\,\,-\frac{1}{p}\int_{\pa_{i}B_{n}}|\nabla u_{n}|^{p}x\cdot\nu d\sigma\vspace{4mm}\\
 & \,\,\,+\int_{\pa_{i}B_{n}}x\cdot\nu\Big[\frac{1}{p_{n}}|u_{n}|^{p_{n}}+\frac{1}{p_{n}(s)}\frac{|u_{n}|^{p_{n}(s)}}{|x|^{s}}+\frac{1}{p}a(x)|u_{n}|^{p}\Big]d\sigma.\\
 & =:R_{1}.
\end{aligned}
\label{eq: 4.6}
\end{equation}
By assumption (\ref{condition of WX}), we have
\[
\int_{\partial_{e}B_{n}}|\nabla u_{n}|^{p}x\cdot\nu d\sigma\le0.
\]
Also note that $a(0)>0$. Thus (\ref{eq: 4.6}) gives us
\begin{equation}
L_{1}\ge\frac{1}{2}a(0)\int_{B_{n}}|u_{n}|^{p}dx.\label{eq: 4.7.1}
\end{equation}
On the other hand, since $|x|\le C\la_{n}^{-1/p}$ for $x\in\pa B_{n}$,
by (\ref{eq: 4.2}), we have
\begin{eqnarray}
R_{1} & \le & C\la_{n}^{-\frac{1}{p}}\int_{\pa_{i}B_{n}}\Bigl(|u_{n}|^{p_{n}}+|u_{n}|^{p}+|\nabla u_{n}|^{p}+\frac{|u_{n}|^{p_{n}(s)}}{|x|^{s}}\Bigl)d\sigma\nonumber \\
 &  & +C\int_{\pa_{i}B_{n}}|\na u_{n}|^{p-1}|u_{n}|d\si\label{eq: 4.7.2}\\
 & \le & C\la_{n}^{\frac{p-N}{p}}.\nonumber
\end{eqnarray}
Thus combining (\ref{eq: 4.6}), (\ref{eq: 4.7.1}) and (\ref{eq: 4.7.2})
implies that
\begin{equation}
\int_{B_{n}}|u_{n}|^{p}dx\le C\la_{n}^{\frac{p-N}{p}}.\label{eq: 4.8}
\end{equation}
Now arguing as that of \cite{CaoPengYan2012}, we have
\[
\int_{B_{n}}|u_{n}|^{p}dx\ge C^{\prime}\la_{n}^{-p}.
\]
Therefore we arrive at
\begin{equation}
\la_{n}^{-p}\le C\la_{n}^{\frac{p-N}{p}}.\label{eq: 4.12}
\end{equation}
Since $\la_{n}\to\wq$, (\ref{eq: 4.12}) can not happen under the
assumption that
\[
N>p^{2}+p.
\]

The case $0\in\Om$ turns out to be easier than the previous case
since $\pa_{e}B_{n}=\emptyset$ now. So (\ref{eq: 4.6}) holds as
well with $\pa_{i}B_{n}=\pa B_{n}$. Arguing as above, we get a contradiction.
So we complete the proof of Theorem \ref{thm: main result} in Case
1.

Now we consider Case 2. That is, $x^{*}\ne0$. We have two possibilities:
either $B_{t_{n}\lambda_{n}^{-\frac{1}{p}}}(x_{n})\subset\subset\Omega$
or $B_{t_{n}\lambda_{n}^{-\frac{1}{p}}}(x_{n})\cap(\R^{N}\backslash\Omega)\neq\emptyset$.

Suppose that $B_{t_{n}\lambda_{n}^{-\frac{1}{p}}}(x_{n})\subset\subset\Omega$.
Then $B_{n}=B_{t_{n}\lambda_{n}^{-\frac{1}{p}}}(x_{n})$. We take
$x_{0}=x_{n}$ in (\ref{eq: 4.4}) and obtain that
\begin{equation}
\begin{aligned}L_{2}:= & \frac{s}{p_{n}(s)}\int_{B_{n}}\frac{|u_{n}|^{p_{n}(s)}}{|x|^{2+s}}(x_{n}\cdot x)dx\vspace{4mm}\\
\le & -\int_{B_{n}}\Big[a(x)-\frac{1}{p}\nabla a(x)\cdot(x-x_{n})\Big]|u_{n}|^{p}dx\\
 & +\frac{N-p}{p}{\displaystyle \int_{\pa B_{n}}|\nabla u_{n}|^{p-2}\frac{\partial u_{n}}{\partial\nu}u_{n}d\sigma}+\int_{\pa B_{n}}|\nabla u_{n}|^{p-2}\nabla u_{n}\cdot(x-x_{n})\frac{\partial u_{n}}{\partial\nu}d\sigma\vspace{4mm}\\
 & \,\,\,-\frac{1}{p}\int_{\pa B_{n}}|\nabla u_{n}|^{p}(x-x_{n})\cdot\nu d\sigma\vspace{4mm}\\
 & \,\,\,+\int_{\pa B_{n}}(x-x_{n})\cdot\nu\Big[\frac{1}{p_{n}}|u_{n}|^{p_{n}}+\frac{1}{p_{n}(s)}\frac{|u_{n}|^{p_{n}(s)}}{|x|^{s}}+\frac{1}{p}a(x)|u_{n}|^{p}\Big]d\sigma.\\
 & =:R_{2}.
\end{aligned}
\label{eq: 4.14}
\end{equation}
Since $x_{n}\to x^{*}$, we have $x_{n}\cdot x\ge\frac{1}{2}|x_{n}|^{2}\ge\frac{1}{4}|x^{*}|^{2}$.
Thus
\[
L_{2}\ge C\int_{B_{n}}|u_{n}|^{p_{n}(s)}dx.
\]
Again, applying the same argument as that of \cite{CaoPengYan2012}
gives us that
\begin{equation}
L_{2}\ge C^{\prime}\la_{n}^{-N+p_{n}(s)\frac{N-p}{p}}.\label{eq: 4.15}
\end{equation}
On the other hand, by arguing as before, we easily get that
\begin{equation}
R_{2}\le C\la_{n}^{\frac{p-N}{p}}+C\int_{B_{n}}|u_{n}|^{p}dx,\label{eq: 4.16}
\end{equation}
in which the assumption $a\in C^{1}(\bar{\Om})$ was used. We claim
that
\begin{equation}
\int_{B_{n}}|u_{n}|^{p}dx\le C\la_{n}^{-p}.\label{eq: 4.17}
\end{equation}
Indeed, let $p_{1}>p^{*}$ such that $\frac{N}{p}(1-\frac{p}{p_{1}})>p$.
This is possible since $N>p^{2}+p$. Also, let $p_{2}=p$. Then we
have $p^{*}/p^{\prime}<p_{2}<p^{*}$. By proposition \ref{prop: main result in section 2},
there exist $v_{i}\geq0,i=1,2,$ such that $|u_{n}|\leq v_{1}+v_{2}$
and
\[
\|v_{1}\|_{*,p_{1}}\leq C,\,\,\,\|v_{2}\|_{*,p}\leq C\lambda_{n}^{\frac{N-p}{p}-\frac{N}{p}}=C\la_{n}^{-1}.
\]
Hence
\[
\int_{B_{n}}|u_{n}|^{p}dx\leq2^{p-1}\int_{B_{n}}|v_{1}|^{p}dx+2^{p-1}\int_{B_{n}}|v_{2}|^{p}dx\leq C\lambda_{n}^{-\frac{N}{p}(1-\frac{p}{p_{1}})}+C\lambda_{n}^{-p}\leq C\lambda_{n}^{-p}.
\]
This gives (\ref{eq: 4.17}). Now combining (\ref{eq: 4.15})-(\ref{eq: 4.17})
gives us
\begin{equation}
\lambda_{n}^{-N+\frac{N-p}{p}p_{n}(s)}\leq C\lambda_{n}^{-p}+C\lambda_{n}^{\frac{p-N}{p}}\leq C\lambda_{n}^{-p},\label{eq: 4.18}
\end{equation}
which is impossible since $N>p^{2}+p$ and $s<p$.

It remains to consider $B_{t_{n}\lambda_{n}^{-\frac{1}{p}}}(x_{n})\cap(\R^{N}\setminus\Omega)\neq\emptyset$.
In (\ref{eq: 4.4}), we take $x_{0}\in\R^{N}\setminus\Omega$ with
$|x_{0}-x_{n}|\leq2t_{n}\lambda_{n}^{-\frac{1}{p}}$ and $\nu\cdot(x-x_{0})\leq0$
in $\partial\Omega\cap B_{n}.$ With this choice of $x_{0},$ we get
from (\ref{eq: 4.4}),
\[
\begin{aligned} & \frac{s}{p_{n}(s)}\int_{B_{n}}\frac{|u_{n}|^{p_{n}(s)}}{|x|^{2+s}}(x_{n}\cdot x)dx\vspace{4mm}\\
\le & -\int_{B_{n}}\Big[a(x)-\frac{1}{p}\nabla a(x)\cdot(x-x_{n})\Big]|u_{n}|^{p}dx\\
 & +\frac{N-p}{p}{\displaystyle \int_{\pa_{i}B_{n}}|\nabla u_{n}|^{p-2}\frac{\partial u_{n}}{\partial\nu}u_{n}d\sigma}+\int_{\pa_{i}B_{n}}|\nabla u_{n}|^{p-2}\nabla u_{n}\cdot(x-x_{n})\frac{\partial u_{n}}{\partial\nu}d\sigma\vspace{4mm}\\
 & \,\,\,-\frac{1}{p}\int_{\pa_{i}B_{n}}|\nabla u_{n}|^{p}(x-x_{n})\cdot\nu d\sigma\vspace{4mm}\\
 & \,\,\,+\int_{\pa_{i}B_{n}}(x-x_{n})\cdot\nu\Big[\frac{1}{p_{n}}|u_{n}|^{p_{n}}+\frac{1}{p_{n}(s)}\frac{|u_{n}|^{p_{n}(s)}}{|x|^{s}}+\frac{1}{p}a(x)|u_{n}|^{p}\Big]d\sigma.
\end{aligned}
\]
Arguing as before, we find that (\ref{eq: 4.18}) still holds. Thus
we get a contradiction. We complete the proof of Theorem \ref{thm: main result}.\end{proof}

Now we can prove Theorem \ref{thm: infinitely many solutions}.

\begin{proof}[Proof of Theorem \ref{thm: infinitely many solutions}]
With Theorem \ref{thm: main result} at hand, we can prove Theorem
\ref{thm: infinitely many solutions} by the same method as that of
\cite{CaoPengYan2012}. So we omit the details.  \end{proof}

\appendix

\section{Estimates for quasilinear problems with Hardy potential}

In this section, we deduce some elementary estimates for solutions
of a quasilinear elliptic problem involving a Hardy potential. Let
$D$ be a bounded domain in $\R^{N}$ and $0\in D.$ For any $0\le t<p$,
write $d\mu_{t}=|x|^{-t}dx$ and $\|w\|_{q,\mu_{t}}^{q}=\int_{D}|w|^{q}d\mu_{t}.$
We also use the notation $\|w\|_{q}=\|w\|_{q,\mu_{0}}.$ Let us recall
that
\[
\|w\|_{*,q}=\|w\|_{q}+||w\|_{\frac{(N-s)}{N}q,\mu_{s}}.
\]

\begin{prop}
\label{prop: A.1} For any $f_{_{i}}\geq0$ and $f_{i}\in L^{\infty}({D}),i=1,2,$
let $w\in W_{0}^{1,p}(D)$ be the solution of
\[
\begin{cases}
-\Delta_{p}w=f_{1}(x)+\frac{f_{2}(x)}{|x|^{s}}, & x\in D,\\
w=0, & \text{on }\pa D.
\end{cases}
\]
Then, for any $1<q<N/p$, there is $C=C(N,p,s,q)>0$ such that
\[
\|w\|_{\ast,\frac{(p-1)Nq}{N-pq}}\le C\left(\|f_{1}\|_{q}+\|f_{2}\|_{\frac{(N-s)q}{N-sq},\mu_{s}}\right){}^{\frac{1}{p-1}}.
\]
\end{prop}
\begin{proof}
By the maximum principle, we find that $w\geq0.$ We claim that if
$\ensuremath{r>1/p',}$ then
\begin{equation}
\lv w\rv_{\ast,p^{*}r}^{pr}\le C\int_{D}\left(f_{1}+\frac{f_{2}}{|x|^{s}}\right)w^{1+p(r-1)}dx,\label{eq: a.1}
\end{equation}
for some $C=C(r)>0$.

First we suppose that $r\ge1$. Since $f_{1},f_{2}$ are bounded functions,
it is standard to prove that $w\in L^{\infty}(D)$ by Moser's iteration
method \cite{Moser1960}. Then we can take a test function $\xi=w^{1+p(r-1)}$
so that
\begin{equation}
\frac{1+p(r-1)}{r^{p}}\int_{D}|\na w^{r}|^{p}dx=\int_{D}\left(f_{1}+\frac{f_{2}}{|x|^{s}}\right)w^{1+p(r-1)}dx.\label{eq: a.2}
\end{equation}
Applying Sobolev inequality and Caffarelli-Kohn-Nirenberg inequality
give us
\[
\lv w\rv_{p^{*}r}^{pr}+\lv w\rv_{p^{*}(s)r}^{pr}\le C\int_{D}|\na w^{r}|^{p}dx
\]
for some $C=C(N,p,s)>0$. Thus
\begin{equation}
\lv w\rv_{\ast,p^{*}r}^{pr}\le C\int_{D}|\na w^{r}|^{p}dx.\label{eq: a.3}
\end{equation}
Therefore, combining (\ref{eq: a.2}) and (\ref{eq: a.3}) yields
(\ref{eq: a.1}).

Now let $r\in(1/p^{\prime},1)$ and $\ep>0$. Define $\xi=w(w+\ep)^{p(r-1)}$.
It is easy to verify that $\xi\in W_{0}^{1,p}(D)$ and
\[
\na\xi=(w+\epsilon)^{p(r-1)}\na w+p(r-1)w(w+\epsilon)^{p(r-1)-1}\na w.
\]
Take $\xi$ as a test function. We have
\[
\int_{D}|\na w|^{p-2}\na w\cdot\na\xi dx=\int_{D}\left(f_{1}+\frac{f_{2}}{|x|^{s}}\right)\xi dx.
\]
A simple calculation gives that
\begin{eqnarray*}
\int_{D}|\na w|^{p-2}\na w\cdot\na\xi dx & \ge & (1+p(r-1))\int_{D}(w+\ep)^{p(r-1)}|\na w|^{p}\\
 & = & \frac{1+p(r-1)}{r^{p}}\int_{D}|\na\left((w+\ep)^{r}-\ep^{r}\right)|dx\\
 & \ge & C(r)\left(\lv\left((w+\ep)^{r}-\ep^{r}\right)\rv_{p^{*}}^{p}+\lv\left((w+\ep)^{r}-\ep^{r}\right)\rv_{p^{*}(s),\mu_{s}}^{p}\right),
\end{eqnarray*}
for $C=C(N,p,s,r)>0$.

Let $w_{\ep}=\left((w+\ep)^{r}-\ep^{r}\right)^{1/r}$. Then there
exists $C>0$ such that
\[
\int_{D}|\na w|^{p-2}\na w\cdot\na\xi dx\ge C(r)\lv w_{\ep}\rv_{\ast,p^{*}r}^{pr}.
\]
Thus
\[
\lv w_{\ep}\rv_{\ast,p^{*}r}^{pr}\le C\int_{D}\left(f_{1}+\frac{f_{2}}{|x|^{s}}\right)w(w+\ep)^{p(r-1)}dx.
\]
Letting $\ep\to0$, we obtain (\ref{eq: a.1}) in the case when $r\in(1/p^{\prime},1)$.

To prove Proposition \ref{prop: A.1}, we apply H\"older's inequality
to (\ref{eq: a.1}) and obtain that
\[
\begin{aligned}\lv w\rv_{\ast,p^{*}r}^{pr} & \le C\left(\lv f_{1}\rv_{\frac{p^{*}r}{p^{*}r-(1+p(r-1))}}+\lv f_{2}\rv_{\frac{p^{*}(s)r}{p^{*}(s)r-(1+p(r-1))},\mu_{s}}\right)\left(\lv w\rv_{p^{*}r}^{1+p(r-1)}+\lv w\rv_{p^{*}(s)r,\mu_{s}}^{1+p(r-1)}\right)\\
 & \le C\left(\lv f_{1}\rv_{\frac{p^{*}r}{p^{*}r-(1+p(r-1))}}+\lv f_{2}\rv_{\frac{p^{*}(s)r}{p^{*}(s)r-(1+p(r-1))},\mu_{s}}\right)\lv w\rv_{\ast,p^{*}r}^{1+p(r-1)},
\end{aligned}
\]
which implies that
\[
\lv w\rv_{\ast,p^{*}r}\le C\left(\lv f_{1}\rv_{\frac{p^{*}r}{p^{*}r-(1+p(r-1))}}+\lv f_{2}\rv_{\frac{p^{*}(s)r}{p^{*}(s)r-(1+p(r-1))},\mu_{s}}\right)^{\frac{1}{p-1}}.
\]

Give $q\in(1,N/p)$. Let $r\in(1/p^{\prime},\wq)$ be such that $q=\frac{p^{*}r}{p^{*}r-(1+p(r-1))}$.
Then a simple calculation gives us
\begin{eqnarray*}
\frac{p^{*}(s)r}{p^{*}(s)r-(1+p(r-1))}=\frac{(N-s)q}{N-sq} & \text{ and } & p^{*}r=\frac{(p-1)Nq}{N-pq}.
\end{eqnarray*}
 We finish the proof.
\end{proof}
As a consequence of Proposition \ref{prop: A.1} we have the following
corollary.
\begin{cor}
\label{cor: A.2} Let $w\in W_{0}^{1,p}(D)$ be the solution of
\[
\begin{cases}
-\Delta_{p}w=\left(a_{1}(x)+\frac{a_{2}(x)}{|x|^{s}}\right)v^{p-1}, & x\in D,\\
w=0, & \text{on }\pa D,
\end{cases}
\]
where $a_{1},a_{2},v\in L^{\wq}(D)$ are nonnegative functions. Then
for any $\wq>q>p^{*}/p^{\prime}$, there holds
\[
\|w\|_{*,q}\le C\left(\|a_{1}\|_{\frac{N}{p}}+\|a_{2}\|_{\frac{N-s}{p-s},\mu_{s}}\right)^{\frac{1}{p-1}}\|v\|_{*,q}
\]
for $C=C(N,p,s,q)>0$.\end{cor}
\begin{proof}
Let $\wq>q>p^{*}/p^{\prime}$ and define $r=Nq/(N(p-1)+pq)$. Then
$1<r<N/p$ and $q=(p-1)Nr/(N-pr)$.

By applying Proposition \ref{prop: A.1} with $f_{i}=a_{i}v^{p-1}$,
$i=1,2$, we obtain that
\[
\|w\|_{*,q}\le C\left(\|f_{1}\|_{r}+||f_{2}\|_{\frac{(N-s)r}{N-sr},\mu_{s}}\right){}^{\frac{1}{p-1}},
\]
for $C=C(N,p,s,q)>0$. By H\"older's inequality and the definition
of $\|\cdot\|_{\ast,q}$, we have that
\[
\|f_{1}\|_{r}\le\|a_{1}\|_{\frac{N}{p}}\|v\|_{q}^{p-1}\le\|a_{1}\|_{\frac{N}{p}}\|v\|_{\ast,q}^{p-1}
\]
and that
\[
||f_{2}||_{\frac{(N-s)r}{N-sr},\mu_{s}}\le\|a_{2}\|_{\frac{N-s}{p-s},\mu_{s}}\|v\|_{\frac{(N-s)q}{N},\mu_{s}}^{p-1}\le\|a_{2}\|_{\frac{N-s}{p-s},\mu_{s}}\|v\|_{\ast,q}^{p-1}.
\]
Combining the above inequalities gives Corollary \ref{cor: A.2}.
\end{proof}
We also have the following corollary.
\begin{cor}
\label{cor: A.3} Let $w\in W_{0}^{1,p}(D)$ be the solution of
\[
\begin{cases}
-\Delta_{p}w=\left(a_{1}(x)+\frac{a_{2}(x)}{|x|^{s}}\right)v^{p-1}, & x\in D,\\
w=0, & \text{on }\pa D,
\end{cases}
\]
where $a_{1},a_{2},v\in L^{\wq}(D)$ are nonnegative functions. Then
for any $p_{2}\in(p^{*}/p^{\prime},p^{*})$, there is a constant $C=C(N,p,s,p_{2})>0$
such that
\[
\|w\|_{*,p_{2}}\le C\left(\|a_{1}\|_{r_{1}}+\|a_{2}\|_{r_{2},\mu_{s}}\right)^{\frac{1}{p-1}}\|v\|_{*,p^{*}},
\]
where $r_{1},r_{2}$ are defined by
\begin{eqnarray}
\frac{1}{r_{1}}=(p-1)\left(\frac{1}{p_{2}}-\frac{1}{p^{*}}\right)+\frac{p}{N}, & \text{ and } & \frac{1}{r_{2}}=(p-1)\left(\frac{N}{(N-s)p_{2}}-\frac{1}{p^{*}(s)}\right)+\frac{p-s}{N-s}.\label{eq: A.3.1}
\end{eqnarray}
\end{cor}
\begin{proof}
The proof is similar to that of Corollary \ref{cor: A.2}. By applying
Proposition \ref{prop: A.1} with $f_{i}=a_{i}v^{p-1}$, $i=1,2$,
we obtain
\[
\|w\|_{*,p_{2}}\le C\left(\|f_{1}\|_{\frac{Np_{2}}{(p-1)N+p_{2}p}}+\|f_{2}\|_{\frac{(N-s)p_{2}}{(p-1)N+(p-s)p_{2}},\mu_{s}}\right){}^{\frac{1}{p-1}}.
\]
Define $r_{1},r_{2}$ by (\ref{eq: A.3.1}). Applying H\"older's inequality
gives us that
\[
\|f_{1}\|_{\frac{Np_{2}}{(p-1)N+p_{2}p}}\le\|a_{1}\|_{r_{1}}\|v\|_{p^{*}}^{^{p-1}},
\]
 and that
\[
\|f_{2}\|_{\frac{(N-s)p_{2}}{(p-1)N+(p-s)p_{2}},\mu_{s}}\le\|a_{2}\|_{r_{2},\mu_{s}}\|v\|_{p^{*}(s),\mu_{s}}^{^{p-1}}.
\]
Combining the above inequalities gives Corollary \ref{cor: A.3}.
\end{proof}
We will need the following lemma in Section 3.
\begin{lem}
\label{lem: A.4}Let $w\in W_{\text{loc }}^{1,p}(\R^{N})$, $w\ge0$
be a weak solution of the equation
\[
-\De_{p}w\le\left(a_{1}(x)+\frac{a_{2}(x)}{|x|^{s}}\right)w^{p-1}
\]
in $\R^{N}$, where $a_{1},a_{2}\in L_{\text{loc}}^{\infty}(\R^{N})$
are nonnegative functions. Then for any unit ball $B_{1}(y)\subset\R^{N}$
and for any $q>p^{*}$, there is a small constant $\de=\de(q)>0$
such that if
\[
\left(\int_{B_{1}(y)}a_{1}^{\frac{N}{p}}dx\right)^{\frac{p}{N}}+\left(\int_{B_{1}(y)}a_{2}^{\frac{N-s}{p-s}}d\mu_{s}\right)^{\frac{p-s}{N-s}}<\de,
\]
then for any $\ga\in(0,p^{*})$, there has
\[
||w||_{q,B_{1/2}(y)}\le C||w||_{\ga,B_{1}(y)}
\]
for some $C=C(N,p,s,q,\ga)>0$. \end{lem}
\begin{proof}
It is standard to show that $w\in L_{\text{loc}}^{\infty}(\R^{N})$
by Moser's iteration method \cite{Moser1960}. Thus for any $\eta\in C_{0}^{\infty}(B_{1}(y))$,
we can take a test function by $\var=\eta^{p}w^{1+p(\tau-1)}$ for
any $\tau\ge1$. Write $B_{r}=B_{r}(y)$ for $r>0$ in the following
proof. Then we have
\begin{equation}
\int_{B_{1}}|\na w|^{p-2}\na w\cdot\na\var dx\le\int_{B_{1}}\left(a_{1}(x)+\frac{a_{2}(x)}{|x|^{s}}\right)\eta^{p}w^{p\tau}dx.\label{eq: A.4.1}
\end{equation}
Firstly, we have
\[
\int_{B_{1}}|\na w|^{p-2}\na w\cdot\na\var dx\ge\frac{C}{\tau^{p-1}}\int_{B_{1}}|\na(\eta w)|^{p}dx-C\int_{B_{1}}|\na\eta|^{p}w^{p\tau}dx.
\]
Secondly, we have
\[
\int_{B_{1}}a_{1}(x)\eta^{p}w^{p\tau}dx\le\left(\int_{B_{1}}a_{1}^{\frac{N}{p}}dx\right)^{\frac{p}{N}}\left(\int_{B_{1}}(\eta w^{\tau})^{p^{*}}dx\right)^{\frac{p}{p^{*}}},
\]
and
\[
\int_{B_{1}}a_{2}(x)\eta^{p}w^{p\tau}d\mu_{s}\le\left(\int_{B_{1}}a_{2}^{\frac{N-s}{p-s}}d\mu_{s}\right)^{\frac{p-s}{N-s}}\left(\int_{B_{1}}(\eta w^{\tau})^{p^{*}(s)}d\mu_{s}\right)^{\frac{p}{p^{*}(s)}}.
\]
Thus (\ref{eq: A.4.1}) implies that
\begin{equation}
\begin{aligned}\int_{B_{1}}|\na(\eta w)|^{p}dx & \le C\int_{B_{1}}|\na\eta|^{p}w^{p\tau}dx\\
 & \;+CA\left(\left(\int_{B_{1}}(\eta w^{\tau})^{p^{*}}dx\right)^{\frac{p}{p^{*}}}+\left(\int_{B_{1}}(\eta w^{\tau})^{p^{*}(s)}d\mu_{s}\right)^{\frac{p}{p^{*}(s)}}\right),
\end{aligned}
\label{eq: A.4.2}
\end{equation}
 where $C=C(\tau)>0$ and $A$ is given by
\[
A=\left(\int_{B_{1}}a_{1}^{\frac{N}{p}}dx\right)^{\frac{p}{N}}+\left(\int_{B_{1}}a_{2}^{\frac{N-s}{p-s}}d\mu_{s}\right)^{\frac{p-s}{N-s}}.
\]
By Sobolev inequality and Caffarelli-Kohn-Nirenberg inequality, we
obtain that
\begin{equation}
\left(\int_{B_{1}}(\eta w^{\tau})^{p^{*}}dx\right)^{\frac{p}{p^{*}}}+\left(\int_{B_{1}}(\eta w^{\tau})^{p^{*}(s)}d\mu_{s}\right)^{\frac{p}{p^{*}(s)}}\le C(N,p,s)\int_{B_{1}}|\na(\eta w)|^{p}dx.\label{eq: A.4.3}
\end{equation}
Combining (\ref{eq: A.4.2}) and (\ref{eq: A.4.3}) yields that
\[
\begin{aligned} & \left(\int_{B_{1}}(\eta w^{\tau})^{p^{*}}dx\right)^{\frac{p}{p^{*}}}+\left(\int_{B_{1}}(\eta w^{\tau})^{p^{*}(s)}d\mu_{s}\right)^{\frac{p}{p^{*}(s)}}\\
\le & \; C\int_{B_{1}}|\na\eta|^{p}w^{p\tau}dx+CA\left(\left(\int_{B_{1}}(\eta w^{\tau})^{p^{*}}dx\right)^{\frac{p}{p^{*}}}+\left(\int_{B_{1}}(\eta w^{\tau})^{p^{*}(s)}d\mu_{s}\right)^{\frac{p}{p^{*}(s)}}\right).
\end{aligned}
\]
Thus we can choose
\begin{equation}
\de=\de(\tau)>0\label{eq: delta}
\end{equation}
 small enough such that if $A<\de$, then $CA<1/2$ and
\[
\left(\int_{B_{1}}(\eta w^{\tau})^{p^{*}}dx\right)^{\frac{p}{p^{*}}}+\left(\int_{B_{1}}(\eta w^{\tau})^{p^{*}(s)}d\mu_{s}\right)^{\frac{p}{p^{*}(s)}}\le C\int_{B_{1}}|\na\eta|^{p}w^{p\tau}dx.
\]
In particular, if $A<\de,$ we have
\begin{equation}
\left(\int_{B_{1}}(\eta w^{\tau})^{p^{*}}dx\right)^{\frac{p}{p^{*}}}\le C(\tau)\int_{B_{1}}|\na\eta|^{p}w^{p\tau}dx.\label{eq: A.4.4}
\end{equation}

Let $0<r<R\le1$ and $\eta\in C_{0}^{\infty}(R_{R})$ be a cut-off
function such that $0\le\eta\le1$, $\eta\equiv1$ in $B_{r}$ and
$|\na\eta|\le2/(R-r)$. Substituting $\eta$ into (\ref{eq: A.4.4})
gives us that
\begin{equation}
\left(\int_{B_{r}}w^{p\chi\tau}dx\right)^{\frac{1}{\chi}}\le\frac{C(\tau)}{(R-r)^{p}}\int_{B_{R}}w^{p\tau}dx,\label{eq: A.4.5}
\end{equation}
where $\chi=p^{*}/p>1$.

Now for any fixed $q>p^{*}$, there exists $k\in\N$ such that $p\chi^{k}\le q<p\chi^{k+1}$.
Let $\tau_{i}=\chi^{i}$, $i=1,\ldots,k$ and let
\[
\de=\min\{\de(\tau_{i})\}_{i=1}^{k},
\]
where $\de(\tau_{i})$ is defined by (\ref{eq: delta}) with $\tau=\tau_{i}$.
Then if $A<\de$, we obtain from (\ref{eq: A.4.5}) that, for all
$\tau=\tau_{i}$, $i=1,\ldots,k$,
\[
\left(\int_{B_{r}}w^{p\chi^{i+1}}dx\right)^{\frac{1}{p\chi^{i+1}}}\le\frac{C(\tau_{i})}{(R-r)^{1/\chi^{i}}}\left(\int_{B_{R}}w^{p\chi^{i}}dx\right)^{\frac{1}{p\chi^{i}}}.
\]
Let $r_{i}=r+(R-r)/2^{i-1}$, $i\ge1$. Take $r=r_{i},R=r_{i-1}$
in the above formula and iterate for finitely many times. We obtain,
for any $0<r<R\le1$,
\[
\left(\int_{B_{r}}w^{p\chi^{k+1}}dx\right)^{\frac{1}{p\chi^{k+1}}}\le\frac{C}{(R-r)^{\si}}\left(\int_{B_{R}}w^{p^{*}}dx\right)^{\frac{1}{p^{*}}}
\]
for some constants $C>0$ and $\si>0$. In particular, we have
\begin{equation}
\left(\int_{B_{r}}w^{q}dx\right)^{\frac{1}{q}}\le\frac{C}{(R-r)^{\si}}\left(\int_{B_{R}}w^{p^{*}}dx\right)^{\frac{1}{p^{*}}}.\label{eq:A.4.6}
\end{equation}

Fix $\ga\in(0,p^{*})$. There exists $\theta\in(0,1)$ such that
\[
\frac{1}{p^{*}}=\frac{\theta}{\ga}+\frac{1-\theta}{q}.
\]
Thus by H\"older's inequality and Young's inequality, (\ref{eq:A.4.6})
implies that
\[
\left(\int_{B_{r}}w^{q}dx\right)^{\frac{1}{q}}\le\frac{1}{2}\left(\int_{B_{R}}w^{q}dx\right)^{\frac{1}{q}}+\frac{C}{(R-r)^{\si/\theta}}\left(\int_{B_{R}}w^{\ga}dx\right)^{\frac{1}{\ga}}.
\]
Now an iteration argument gives us that
\[
\left(\int_{B_{r}}w^{q}dx\right)^{\frac{1}{q}}\le\frac{C}{(R-r)^{\si^{\prime}}}\left(\int_{B_{R}}w^{\ga}dx\right)^{\frac{1}{\ga}}
\]
for some constants $C,\si^{\prime}>0$. Choose $r=1/2$ and $R=1$.
We complete the proof.
\end{proof}

\section{A decay estimate}

We use $\R_{*}^{N}$ to denote either $\R^{N}$ or $\R_{+}^{N}.$
Consider the following equation
\begin{equation}
\begin{cases}
-\Delta_{p}u=\mu|u|^{p^{*}-2}u+\frac{|u|^{p^{*}(s)-2}u}{|x|^{s}}, & \text{in}\;\R_{*}^{N},\\
u\in{\cal D}_{0}^{1,p}(\R_{*}^{N}),
\end{cases}\label{b.1}
\end{equation}
where ${\cal D}_{0}^{1,p}(\R_{*}^{N})$ is the completion of $C_{0}^{\wq}(\R_{\ast}^{N})$
in the norm $\|u\|_{{\cal D}_{0}^{1,p}(\R_{*}^{N})}=\|\na u\|_{p,\R_{\ast}^{N}}$.
In this section, we give an estimate for the decay of solutions to
equation \eqref{b.1} at the infinity.  We have the following result.
\begin{prop}
\label{prop: B.1} Let $u$ be a solution of \eqref{b.1}. Then there
exists a constant $C>0$ such that
\begin{eqnarray*}
|u(x)|\leq\frac{C}{1+|x|^{\frac{N-p}{p-1}}}, &  & \forall\, x\in\R_{\ast}^{N}.
\end{eqnarray*}

\end{prop}
To prove Proposition \ref{prop: B.1}, the following preliminary estimate
is needed.
\begin{lem}
\label{lem: B.2} Let $u$ be a solution of \eqref{b.1}. Then there
is a constant $C>0$ such that
\begin{eqnarray*}
|u(x)|\leq\frac{C}{1+|x|^{\frac{N-p}{p}+\si}}, &  & \forall\,|x|\ge1,
\end{eqnarray*}
 for some $\si>0$.
\end{lem}
Lemma \ref{lem: B.2} can be proved as that of \cite[Lemma B.3]{CaoPengYan2012}
or \cite[Proposition 2.1]{CXY}. So we omit the details. We also need
the following comparison principle which is a special case of \cite[Theorem 1.5]{CXY}.
\begin{thm}
\label{thm: Comparison at infinity} Let $\Omega$ be an exterior
domain such that $\Omega^{c}=\mathbb{R}^{N}\backslash\Omega$ is bounded
and $f\in L^{\frac{N}{p}}(\Omega)$. Let $u\in\mathcal{D}^{1,p}(\Omega)$
be a subsolution of equation
\begin{eqnarray}
-\De_{p}u=f|u|^{p-2}u &  & \hbox{in}\:\Omega,\label{eq: subsolution 2}
\end{eqnarray}
and $v\in\mathcal{D}^{1,p}(\Omega)$  a positive supersolution of
\begin{eqnarray}
-\De_{p}v=g|v|^{p-2}v &  & \hbox{in}\:\Omega,\label{eq: supersolution 2}
\end{eqnarray}
such that $\inf_{\pa\Om}v>0$, where functions $g$ belongs to $L^{\frac{N}{p}}(\Omega)$
and $f\leq g$ in $\Omega$. Moreover, assume that
\begin{equation}
\limsup_{R\to\infty}\frac{1}{R}\int_{B_{2R}\backslash B_{R}}u^{p}|\nabla\log v|^{p-1}=0.\label{eq: special condition on v}
\end{equation}
 If $u\leq v$ on $\partial\Omega$, then
\begin{eqnarray*}
u\leq v &  & \hbox{in}\:\Omega.
\end{eqnarray*}

\end{thm}
Now we can prove Proposition \ref{prop: B.1}.

\begin{proof}[\textbf{Proof of Proposition \ref{prop: B.1}}] Let
$u$ be a weak solution to equation (\ref{b.1}). In case $\R_{\ast}^{N}=\R_{+}^{N}$,
we define an odd extension of $u$ by
\[
\tilde{u}(x)=\begin{cases}
u(x^{\prime},x_{N}) & \text{if }x_{N}\ge0;\\
-u(x^{\prime},-x_{N}) & \text{if }x_{N}<0
\end{cases}
\]
 for $x=(x^{\prime},x_{N})\in\R^{N}$. Then it is direct to verify
that $\tilde{u}\in{\cal D}^{1,p}(\R^{N})$ and $\tilde{u}$ is a solution
of equation (\ref{b.1}) in the whole space $\R^{N}$. Thus in the
rest of the proof we assume that $\R_{\ast}^{N}=\R^{N}$.

We use Theorem \ref{thm: Comparison at infinity} to prove Proposition
\ref{prop: B.1}. Let $\ep>0$ and denote $\ga=(N-p)/(p-1)$. Let
$v(x)=|x|^{-\ga}(1+|x|^{-\ep})$ for $x\ne0$. A simple calculation
gives that
\begin{eqnarray*}
-\De_{p}v=g(x)v^{p-1}, &  & \text{for }x\ne0,
\end{eqnarray*}
where
\[
g(x)=\frac{(p-1)(\ga+\ep)^{p}-(N-p)(\ga+\ep)^{p-1}}{(1+|x|^{-\ep})^{p-1}|x|^{p+(p-1)\ep}}.
\]
Since $(p-1)(\ga+\ep)^{p}-(N-p)(\ga+\ep)^{p-1}>0$ for any $\ep>0$,
it is easy to obtain that
\begin{eqnarray*}
g(x)\ge C|x|^{-p-(p-1)\ep}, &  & \text{for }|x|\ge1,
\end{eqnarray*}
for some constant $C>0$. Thus $g\in L^{\frac{N}{p}}(\R^{N}\backslash B_{1}(0))$
since $\ep>0$.

On the other hand, let $u$ be a solution to equation (\ref{b.1})
and denote
\[
f(x)=\mu|u|^{p^{*}-p}+\frac{|u|^{p^{*}(s)-p}}{|x|^{s}}.
\]
Lemma \ref{lem: B.2} implies that
\begin{eqnarray*}
f(x)\le C|x|^{-\al} &  & \text{for }|x|\ge1,
\end{eqnarray*}
where
\[
\al=\min\left\{ (p^{*}-p)\left(\frac{N-p}{p}+\si\right),s+(p^{*}(s)-p)\left(\frac{N-p}{p}+\si\right)\right\} =p+(p^{*}(s)-p)\si
\]
since $\si>0$, we have $\al>p$, and thus $f\in L^{\frac{N}{p}}(\R^{N}\backslash B_{1}(0))$.

Choose $\ep>0$ small such that $p+(p-1)\ep<\al$. Then we can find
a large number $R>1$ such that
\begin{eqnarray*}
g(x)\ge f(x), &  & \text{for }|x|\ge R.
\end{eqnarray*}
It is easy to verify that the condition (\ref{eq: special condition on v})
is satisfied. Therefore, applying Theorem \ref{thm: Comparison at infinity}
with $\Om=\R^{N}\backslash B_{R}(0)$ gives us
\begin{eqnarray*}
\pm u(x)\le Cv(x), &  & \text{for }|x|\ge R.
\end{eqnarray*}
That is,
\begin{eqnarray*}
|u(x)|\le C|x|^{-\frac{N-p}{p-1}}, &  & \text{for }|x|\ge R.
\end{eqnarray*}

So we obtain the decay rate for the solution $u$ at the infinity.
To prove Proposition \ref{prop: B.1}, one only needs to note that
$u\in L_{\text{loc}}^{\infty}(\R^{N})$, which can be done by Moser's
iteration method \cite{Moser1960}. This finishes the proof. \end{proof}

\section{Estimates for ${\displaystyle p}$-Laplacian equation}

In this section, we copy two results on $p$-Laplacian equation from
\cite{CaoPengYan2012} without proof. We assume that $D$ is a bounded
domain with $\Om\subset\subset D$.
\begin{prop}
\label{prop: C.1} (\cite[Lemma 2.2]{CaoPengYan2012})  For any functions
$f_{1}(x)\geq0$ and $f_{2}(x)\geq0,$ let $w\geq0$ be the solution
of
\[
\begin{cases}
-\De_{p}w=f_{1}+f_{2} & \text{in }D,\\
w=0 & \text{on }\pa D.
\end{cases}
\]
Also, let $w_{i},i=1,2,$ be the solution of
\[
\begin{cases}
-\De_{p}w=f_{i} & \text{in }D,\\
w=0 & \text{on }\pa D,
\end{cases}
\]
respectively. Then, there is a constant $C>0$, depending only on
$r=\frac{1}{3}dist(\Om,\partial{D})$, such that
\begin{eqnarray*}
w(x)\leq C\inf_{y\in B_{r}(x)}w(y)+Cw_{1}(x)+Cw_{2}(x), &  & \forall x\in\Omega.
\end{eqnarray*}

\end{prop}
Next result gives an estimate for solutions of $p$-Laplacian equation
by Wolff potential.
\begin{prop}
\label{prop: C.2}(\cite[Proposition C.1]{CaoPengYan2012}) There
is a constant $\ga\in(p-1,(p-1)N/(N-p+1))$, such that for any solution
$u\in W^{1,p}(D)\cap L^{\wq}(D)$ to equation
\begin{eqnarray*}
-\De_{p}u=f, &  & \text{in }D,
\end{eqnarray*}
where $f\in L^{1}(D)$, $f\ge0$, there exists a constant $C=C(N,p,\ga)>0$,
such that for any $x\in D$ and $r\in(0,\dist(x,\pa D))$,

\[
\left(\fint_{B_{r}(x)}u^{\ga}dy\right)^{\frac{1}{\ga}}\le C+C\int_{r}^{\dist(x,\pa D)}\left(\frac{1}{t^{N-p}}\int_{B_{t}(x)}fdy\right)^{\frac{1}{p-1}}\frac{dt}{t}.
\]

\end{prop}

\section{Global compactness result}

Recall that by (\ref{def: scaling transform}) we define, for any
function $u,$
\[
\rho_{x,\la}(u)=\la^{\frac{N-p}{p}}u(\la(\cdot-x))
\]
 for any $\la>0$ and $x\in\R^{N}$. In this section, we give a global
compactness result in the following proposition.
\begin{prop}
\label{prop:D.1} Let $u_{n}$, $n=1,2,\ldots,$ be a solution of
equation  \eqref{1.3} with $\epsilon=\epsilon_{n}\rightarrow0,$
satisfying $\|u_{n}\|\leq C$ for some constant $C$ independent of
$n$. Then $u_{n}$ can be decomposed as
\[
u_{n}=u_{0}+\sum_{j=1}^{k}\rho_{x_{n,j},\lambda_{n,j}}(U_{j})+\sum_{j=k+1}^{m}\rho_{0,\lambda_{n,j}}(U_{j})+\omega_{n},
\]
where $u_{0}$ is a solution for \eqref{1.1},  $\omega_{n}\rightarrow0$
strongly in $W_{0}^{1,p}(\Omega)$, $x_{n,j}\in\Omega$. And as $n\to\infty$,
$\lambda_{n,j}\rightarrow\infty$ for all $1\le j\le m$, $\lambda_{n,j}d(x_{n,j},\partial\Omega)\rightarrow\infty$
for $j=1,\cdots,k.$

For $j=1,2,\cdot\cdot\cdot,k,$ $U_{j}$ is a solution of
\[
\begin{cases}
-\Delta_{p}u=b_{j}\mu|u|^{p^{*}-2}u, & \text{in }\R^{N},\\
u\in{\cal D}^{1,p}(\R^{N}),
\end{cases}
\]
 for some $b_{j}\in(0,1].$

For $j=k+1,k+2,\cdot\cdot\cdot,m,$ $U_{j}$ is a solution of
\[
\begin{cases}
-\Delta_{p}u=b_{j}\mu|u|^{p^{*}-2}u+b_{j}\frac{|u|^{p^{*}(s)-2}u}{|x|^{s}}, & \text{in }\R_{\ast}^{N},\\
u\in{\cal D}_{0}^{1,p}(\R_{\ast}^{N}),
\end{cases}
\]
 for some $b_{j}\in(0,1],$ where $\R_{*}^{N}=\R^{N}$ if $0\in\Omega,$
while $\R_{*}^{N}=\R_{+}^{N}$ if $0\in\partial\Omega.$

Moreover, set $x_{n,i}=0$ for $i=k+1,\cdots,m.$ For $i,j=1,2,\cdot\cdot\cdot,m,$
if $i\neq j$, then
\[
\frac{\lambda_{n,j}}{\lambda_{n,i}}+\frac{\lambda_{n,i}}{\lambda_{n,j}}+\lambda_{n,j}\lambda_{n,i}|x_{n,i}-x_{n,j}|^{2}\rightarrow\infty
\]
as $n\rightarrow\infty$.\end{prop}
\begin{proof}
The proof is similar to \cite{CaoPengYan2012,CaoYan2010,YanSS1995}
and we omit the details.
\end{proof}
\emph{Acknowledgments. }The authors would like to thank Prof. Xiao
Zhong for his guidance in the preparation of this paper. The first
author is partially supported by NSFC (No.11301204; No. 11371159;
No.11101171) and self-determined research funds of Central China Normal
University from the colleges basic research and operation of MOE(CCNU14A05036).
The second author is financially supported by the Academy of Finland,
project 259224.

\end{document}